\documentclass[11pt]{amsart}
\usepackage{amssymb,amsfonts,amsthm,amscd,amsmath,amsgen,upref}

\usepackage[margin=1in]{geometry}

\theoremstyle{plain}
\newtheorem{cor}{Corollary}

\newtheorem{prop}[cor]{Proposition}
\newtheorem{thm}[cor]{Theorem}
\theoremstyle{definition}

\numberwithin{cor}{section}
\numberwithin{equation}{section}

\DeclareMathOperator{\tr}{tr}

\DeclareMathOperator{\C}{C}
\DeclareMathOperator{\USC}{USC}
\DeclareMathOperator{\LSC}{LSC}
\DeclareMathOperator{\BUC}{BUC}
\DeclareMathOperator{\Lip}{Lip}

\newcommand{\abs}[1]{\lvert#1\rvert}
\newcommand{\norm}[1]{\lVert#1\rVert}

\def\Xint#1{\mathchoice
{\XXint\displaystyle\textstyle{#1}}%
{\XXint\textstyle\scriptstyle{#1}}%
{\XXint\scriptstyle\scriptscriptstyle{#1}}%
{\XXint\scriptscriptstyle\scriptscriptstyle{#1}}%
\!\int}
\def\XXint#1#2#3{{\setbox0=\hbox{$#1{#2#3}{\int}$ }
\vcenter{\hbox{$#2#3$ }}\kern-.6\wd0}}

\def\dashint{\Xint-}

\title{A Partial Homogenization Result for Nonconvex Viscous Hamilton-Jacobi Equations}

\author{Benjamin J. Fehrman}

\date{February 20, 2014}



\address{Department of Mathematics, The University of Chicago, 5734 S. University Avenue, Chicago IL, 60637.}

\email{bfehrman@math.uchicago.edu}

\begin{document}

\begin{abstract}
We provide a general result concerning the homogenization of nonconvex viscous Hamilton-Jacobi equations in the stationary, ergodic setting.  In particular, we show that homogenization occurs for a non-empty set of points within every level set of the effective Hamiltonian, and for every point in the minimal level set of the effective Hamiltonian.  In addition, these methods provide a new proof of homogenization, in full, for convex equations and, for a class of level-set convex equations.  Finally, we prove that the question of homogenization for first order equations reduces to the case that the nonconvexity of the Hamiltonian is localized in the gradient variable.
\end{abstract}

\maketitle

\section{Introduction}

In this paper we present a partial homogenization result for a general class of viscous Hamilton-Jacobi equations with, perhaps, nonconvex nonlinearities.  In order to ease the presentation, we will focus on the first-order case.  However, our methods apply generally to second-order equations following minor, and purely technical alterations.  We similarly restrict the dependencies of the Hamiltonian by assuming a spacial and gradient dependence alone.  This assumption has virtually no effect on the proof, and our results apply readily to more general Hamiltonians.  See (\ref{intro_viscous}) and Section 8 for the details.

Our framework therefore depends on a probability space $(\Omega,\mathcal{F},\mathbb{P})$, which can be viewed as indexing the collection of all equations or environments described by a Hamiltonian $$H(p,y,\omega):\mathbb{R}^d\times\mathbb{R}^d\times\Omega\rightarrow\mathbb{R}.$$  We will in particular assume that the Hamiltonian is uniformly coercive in $p\in\mathbb{R}^d$ and, for each $p\in\mathbb{R}^d$, the process $$(y,\omega)\rightarrow H(p,y,\omega)\;\;\textrm{is stationary and ergodic.}$$  The total list of assumptions are presented precisely in Section 2.

The aim of homogenization techniques, generally, is to characterize the behavior, as $\epsilon\rightarrow 0$, of solutions $u^\epsilon:\mathbb{R}^d\times[0,\infty)\rightarrow\mathbb{R}$ satisfying \begin{equation}\label{intro_eq}\left\{ \begin{array}{ll} u^\epsilon_t+H(Du^\epsilon,x/\epsilon,\omega)=0 & \textrm{on}\;\;\mathbb{R}^d\times(0,\infty), \\ u^\epsilon=u_0 & \textrm{on}\;\;\mathbb{R}^d\times\left\{0\right\},\end{array}\right.\end{equation} for $u_0$ bounded and uniformly continuous on $\mathbb{R}^d$.  Here, the intuition is that, due to the stationarity and ergodicity of the Hamiltonian, as $\epsilon\rightarrow 0$, the behavior of the solution $u^\epsilon$ is representative of an increasing portion of the environments indexed by $\Omega$ and, as $\epsilon\rightarrow 0$, almost surely in $\Omega$, the solutions average to a deterministic limit.

To this end, we introduce, for each $p\in\mathbb{R}^d$, $\delta>0$ and $\omega\in\Omega$, the approximate macroscopic problem \begin{equation}\label{intro_macroscopic} \delta v^\delta+H(p+Dv^\delta,x,\omega)=0\;\;\textrm{on}\;\;\mathbb{R}^d,\end{equation} and identify a deterministic Hamiltonian $\overline{H}:\mathbb{R}^d\rightarrow\mathbb{R}$, \begin{equation}\label{intro_effective}\overline{H}(p)=\limsup_{\delta\rightarrow 0}-\delta v^\delta(0,\omega),\end{equation}  where $\overline{H}(p)$ can be viewed as describing the averaged or homogenized environment.

The expected result is that, if homogenization occurs, as $\epsilon\rightarrow 0$, on a subset of full probability, $$u^\epsilon\rightarrow\overline{u}\;\;\textrm{locally uniformly on}\;\;\mathbb{R}^d\times[0,\infty),$$ for $\overline{u}:\mathbb{R}^d\times[0,\infty)\rightarrow\mathbb{R}$ the unique solution of the well-posed problem \begin{equation}\label{intro_homogenized} \left\{\begin{array}{ll} \overline{u}_t+\overline{H}(D\overline{u})=0 & \textrm{on}\;\;\mathbb{R}^d\times(0,\infty), \\ \overline{u}=u_0 & \textrm{on}\;\;\mathbb{R}^d\times\left\{0\right\}.\end{array}\right.\end{equation}

It is well-known that a homogenization result of this type is essentially equivalent to obtaining, on a subset of full probability, for each $R>0$ and $p\in\mathbb{R}^d$, for $v^\delta$ the solution of (\ref{intro_macroscopic}) corresponding to $p\in\mathbb{R}^d$, \begin{equation}\label{intro_ball} \lim_{\delta\rightarrow 0}\sup_{x\in B_{R/\delta}}\abs{\overline{H}(p)+\delta v^\delta(x,\omega)}=0.\end{equation}  For, if (\ref{intro_ball}) is satisfied, we may obtain (\ref{intro_homogenized}) using the standard perturbed test function method, see Evans \cite{Ev}.

On the contrary, if homogenization holds generally, there exists a subset of full probability such that, for each $p\in\mathbb{R}^d$, the solutions $\tilde{v}^\epsilon:\mathbb{R}^d\rightarrow\mathbb{R}$ satisfying $$\tilde{v}^\epsilon+H(p+D\tilde{v}^\epsilon,x/\epsilon,\omega)=0\;\;\textrm{on}\;\;\mathbb{R}^d,$$ converge, as $\epsilon\rightarrow 0$, locally uniformly to $\overline{v}:\mathbb{R}^d\rightarrow\mathbb{R}$ satisfying $$\overline{v}+\overline{H}(p+D\overline{v})=0\;\;\textrm{on}\;\;\mathbb{R}^d.$$  Since this implies $\overline{v}=-\overline{H}(p)$ is constant, and since uniqueness implies that, for each $\epsilon>0$, $\tilde{v}^\epsilon(x)=\epsilon v^{\epsilon}(x/\epsilon)$, for $v^\epsilon$ the solution of (\ref{macroscopic}) corresponding to $p\in\mathbb{R}^d$, we see that a general homogenization result implies (\ref{intro_ball}).

We present in this paper a general result concerning (\ref{intro_ball}) for first-order and viscous Hamilton-Jacobi equations with nonconvex nonlinearities.  Furthermore, our result provides a new proof of homogenization, in full, in the convex setting.  Here, the effective Hamiltonian is defined as in (\ref{intro_effective}).

\begin{thm}\label{intro_main}  Assume (\ref{steady}).  There exists a subset of full probability such that, for every $p\in\mathbb{R}^d$ satisfying \begin{equation}\label{intro_set} \overline{H}(p)=\min_{q\in\mathbb{R}^d}\overline{H}(q),\;\;\textrm{or,}\;\; p\in\partial\left(\left\{\;\overline{H}(q)\leq\overline{H}(p)\;\right\}\right)\cap\partial\left(\textrm{Conv}\left(\left\{\;\overline{H}(q)\leq\overline{H}(p)\;\right\}\right)\right),\end{equation} for every $R>0$, for $v^\delta$ the solution of (\ref{intro_macroscopic}) corresponding to $p\in\mathbb{R}^d$, $$\lim_{\delta\rightarrow 0}\sup_{x\in B_{R/\delta}}\abs{\overline{H}(p)+\delta v^\delta (x,\omega)}=0.$$\end{thm}

We observe that, for every $p\in\mathbb{R}^d$, the set appearing on the righthand side of (\ref{intro_set}) is never empty.  And, in the case of convex equations, every $p\in\mathbb{R}^d$ satisfies (\ref{intro_set}) which, in view of the above discussion, implies the homogenization of (\ref{intro_eq}).  See the remarks beginning Section 7.

Our methods also apply to viscous equations of the form \begin{equation}\label{intro_viscous}\left\{\begin{array}{ll} u^\epsilon_t-\epsilon \tr(A(x/\epsilon,\omega)D^2u^\epsilon)+H(Du^\epsilon, u^\epsilon,x,x/\epsilon,\omega)=0 & \textrm{on}\;\;\mathbb{R}^d\times(0,\infty), \\ u^\epsilon=u_0 & \textrm{on}\;\;\mathbb{R}^d\times\left\{0\right\},\end{array}\right.\end{equation} and the analogous time-independent problems.  The necessary changes are minor, and purely technical.  The details are outlined in Section 8.

Finally, in Section 9, we provide an application of Theorem \ref{intro_main} to first order equations.  Essentially, we prove that if the nonconvexity of the Hamiltonian $H(p,y,\omega)$ is localized in $p\in\mathbb{R}^d$, then the failure of homogenization is localized to a bounded, open subset of $\mathbb{R}^d$.  Furthermore, we prove that this situation is generic with respect to homogenization.  That is, if homogenization is true for such Hamiltonians, then homogenization is true in general.

The proof relies upon a concrete intuition regarding the so-called metric problem, posed for each $p\in\mathbb{R}^d$ and $\omega\in\Omega$, \begin{equation}\label{intro_metric} H(p+Dm,x,\omega)\leq \overline{H}(p)\;\;\textrm{on}\;\;\mathbb{R}^d\;\;\textrm{with}\;\;H(p+Dm,x,\omega)=\overline{H}(p)\;\;\textrm{on}\;\;\mathbb{R}^d\setminus\left\{0\right\}.\end{equation}  By understanding the asymptotic behavior of specific solutions to (\ref{intro_metric}), we are able to control the convergence of the $-\delta v^\delta$ to $\overline{H}(p)$ at points $p\in\mathbb{R}^d$ satisfying (\ref{intro_set}).  See Section 5 and Section 6 for a complete description.

We remark that a version of (\ref{intro_metric}) was considered in Armstrong and Souganidis \cite{AS} to prove the homogenization of viscous Hamilton-Jacobi equations in unbounded environments, and the first proofs of homogenization for scalar, first-order Hamilton-Jacobi equations in Souganidis \cite{S} and Lions and Souganidis \cite{LS1} were based on the behavior of the first-order, time-dependent version of (\ref{intro_metric}).

Finally, we observe that, prior to this paper, some results were known beyond the convex setting for specialized versions of (\ref{intro_eq}).  In Armstrong and Souganidis \cite{AS2}, homogenization is proved for a special class of first-order, level-set convex equations.  We remark that their assumptions guarantee that (\ref{intro_set}) is satisfied for every $p\in\mathbb{R}^d$.  Furthermore, in Armstrong, Tran and Yu \cite{ATY}, homogenization is obtained for a specific, first-order, nonconvex Hamilton-Jacobi equation.  Here, the analysis relied upon the fact that the nonlinearity could essentially be inverted to yield a family of convex equations, to which convex methods were applied to yield the homogenization of the original equation.

More generally, the homogenization of scalar equations in stationary ergodic random environments has been studied extensively.  The linear case was first analyzed by Papanicolaou and Varadhan \cite{PV,PV1} and Kozlov \cite{K}, and general variational problems were considered by Del Maso and Modica \cite{DM,DM1}.  More recently, results for Hamilton-Jacobi equations were first obtained by \cite{S} and Rezakhanlou and Tarver \cite{RT}, for viscous Hamilton-Jacobi equations by \cite{LS1}, Lions and Souganidis \cite{LS2} and Kosygina, Rezakhanlou and Varadhan \cite{KRV}, for viscous Hamilton-Jacobi equations in unbounded environments by \cite{AS} and for monotone systems of viscous equations in Fehrman \cite{F}.

The paper is organized as follows. In Section 2, we discuss our notation and assumptions in addition to some probabilistic preliminaries.  The approximate macroscopic problem is presented in Section 3, and the definition of the effective Hamiltonian is presented in Section 4.  In Section 5, we analyze the metric problem and obtain, in Section 6, an almost sure characterization of the effective Hamiltonian at points satisfying (\ref{intro_set}).  The proof of homogenization, assuming that every point in space satisfies (\ref{intro_set}), is the topic of Section 7.  In Section 8, we describe the modifications necessary to treat viscous equations of the form (\ref{intro_viscous}) and, in Section 9 present the application to first order equations.  Finally, in the appendix, Section 10, we review some basic facts used in our arguments.

\section{Preliminaries}

\subsection{Notation}

We write $d$ for the spacial dimension.  Elements of $\mathbb{R}^d$ and $[0,\infty)$ are denoted by $x$ and $y$ and $t$ respectively.  We write $Dv$ and $v_t$ for the derivative of the scalar function $v$ with respect to $x\in\mathbb{R}^d$ and $t\in[0,\infty)$, while $D^2v$ stands for the Hessian of $v$.  Regarding the Hamiltonian $H(p,y,\omega)$, we write $p$ for the dependence on the gradient, $y$ for the dependence on $\mathbb{R}^d$ and $\omega$ for the dependence on $\Omega$.  The spaces of $k\times l$ and $k\times k$ symmetric matrices with real entries are respectively written $\mathcal{M}^{k\times l}$ and $\mathcal{S}(k)$.  If $M\in\mathcal{M}^{k\times l}$, then $M^t$ is its transpose and $\abs{M}$ is its norm $\abs{M}=\tr(MM^t)^{1/2}.$  If $M$ is a square matrix, we write $\tr(M)$ for the trace of $M$.  For $U\subset\mathbb{R}^n$, $\USC(U)$, $\LSC(U)$, $\BUC(U)$, $\Lip(U)$ and $\C^k(U)$ are the spaces of upper-semicontinuous, lower-semicontinuous, bounded continuous, Lipschitz continuous and $k$-continuously differentiable functions on $U$ taking values in $\mathbb{R}$.  Moreover, $B_R$ and $B_R(x)$ are respectively the open balls of radius $R$ centered at zero and $x\in\mathbb{R}^d$.  We denote by $\Omega_i\subset\Omega$ various subsets of full probability that we carry throughout the proof.  For a subset $U\subset\mathbb{R}^d$ we write $\partial U$ for the boundary of $U$, we write $\textrm{Conv}\left( U \right)$ for the smallest closed convex set in $\mathbb{R}^d$ containing $U$ and, if $U$ is measurable, we write $\abs{U}$ to denote the measure of $U$. Finally, throughout the paper we write $C$ for constants that may change from line to line but are independent of $\omega\in\Omega$ unless otherwise indicated.

\subsection{The random medium}

The random medium is described by the probability space $(\Omega,\mathcal{F},\mathbb{P})$.  An element $\omega\in\Omega$ can be viewed as indexing a particular realization of the environment.

It is assumed that $\Omega$ is equipped with a group $\left\{\tau_y\right\}_{y\in\mathbb{R}^d}$ of transformations $\tau_y:\Omega\rightarrow\Omega$ which are \begin{equation}\label{transgroup} \textrm{measure-preserving and ergodic,}\end{equation} where the latter means that, if $E\subset\Omega$ satisfies $\tau_y(E)=E$ for each $y\in\mathbb{R}^d$ then $\mathbb{P}(E)=0$ or $\mathbb{P}(E)=1$.

A process $f:\mathbb{R}^d\times\Omega\rightarrow\mathbb{R}$ is said to be stationary if the law of $f(y,\cdot)$ is independent of $y\in\mathbb{R}^d$, a property which can be reformulated using  $\left\{\tau_y\right\}_{y\in\mathbb{R}^d}$ as \begin{equation}\label{statergodic} \begin{array}{ll} f(y+z,\cdot)=f(y,\tau_z\cdot) & \textrm{for all}\;\;\;y,z\in\mathbb{R}^d. \end{array}\end{equation}  To simplify statements we say that a process is stationary ergodic if it satisfies (\ref{statergodic}) and $\left\{\tau_y\right\}_{y\in\mathbb{R}^d}$ is ergodic.

The following ergodic theorem will be used in this paper.  A proof may be found in Becker \cite{B}.  Here $\mathbb{E}f$ denotes the expectation of a random variable $f$.

\begin{thm}\label{ergodic1}  Assume (\ref{transgroup}).  Suppose $f:\mathbb{R}^d\times\Omega\rightarrow\mathbb{R}$ is stationary with $\mathbb{E}\abs{f(0,\cdot)}<\infty$.  Then, on a subset of full probability, for every bounded domain $V\subset\mathbb{R}^d$ containing the origin, $$\lim_{t\rightarrow\infty}\dashint_{tV}f(y,\omega)dy=\mathbb{E}f.$$ \end{thm}

The subadditive ergodic theorem is also used in this paper.  A proof may be found in Akcoglu and Krengel \cite{AK}.  Its statement requires more terminology.  Let $\mathcal{I}$ denote the class of subsets of $[0,\infty)$ consisting of finite unions of intervals of the form $[a,b)$ and let $\left\{\sigma_t\right\}_{t\geq 0}$ be a semigroup of measure-preserving transformations $\sigma_t:\Omega\rightarrow\Omega$.  A map $Q:\mathcal{I}\rightarrow L^1(\Omega,\mathbb{P})$ such that:  \begin{center}\begin{enumerate} \item $Q(I)(\sigma_t\omega)=Q(I+t)\omega\;\;\;\textrm{almost surely in}\;\Omega,$

\item $\mathbb{E}\abs{Q(I)}\leq C\abs{I},$ for some $C>0$ and every $I\in\mathcal{I}$,

\item If $I_1,\ldots,I_k\in\mathcal{I}$ are disjoint then, $Q(\cup_{j=1}^k I_j)\leq \sum_{j=1}^k Q(I_j),$ \end{enumerate} \end{center}  is called a \emph{continuous subadditive process} with respect to the semigroup $\left\{\sigma_t\right\}_{t\geq 0}$.

\begin{thm}\label{subadditiveergodic}  If $Q$ is a continuous subadditive process with respect to the semigroup $\left\{\sigma_t\right\}_{t\geq 0}$, there exists a random variable $a$ which is invariant under $\left\{\sigma_t\right\}_{t\geq 0}$ such that, on a subset of full probability, $$\lim_{t\rightarrow\infty}\frac{1}{t}Q([0,t))(\omega)=a(\omega).$$\end{thm}

\subsection{The assumptions}

We state below a number of assumptions for $H(p,y,\omega)$.  Some are necessary to insure the well-posedness of (\ref{intro_eq}), while others are crucial for the identification of $\overline{H}(p)$ and homogenization of (\ref{intro_eq}).  We make the convention that, unless otherwise indicated, each statement holds globally for $p\in\mathbb{R}^d$, $y\in\mathbb{R}^d$ and $\omega\in\Omega$.

For each fixed $p$, \begin{equation}\label{assumption_stationary} \begin{array}{ll} (y,\omega)\rightarrow H(p,y,\omega) & \textrm{is stationary.} \end{array}\end{equation}

The Hamiltonian is coercive in $p$, i.e., \begin{equation}\label{coercive} \lim_{\abs{p}\rightarrow\infty}\inf_{(y,\omega)\in\mathbb{R}^d\times\Omega}H(p,y,\omega)=+\infty.\end{equation}

The Hamiltonian is bounded for bounded $p$, i.e., for each $R>0$ there exists $C=C(R)>0$ such that, for all $p\in B_R$, \begin{equation}\label{bounded} \abs{H(p,y,\omega)}\leq C, \end{equation} and Lipschitz, locally in the gradient variable, i.e., for each $R>0$ there exists $C=C(R)>0$ satisfying, for all $p_1,p_2\in B_R$, \begin{equation}\label{hamcon}
|H(p_1,y_1,\omega)-H(p_2,y_2,\omega)|<C\left(|p_1-p_2|+\abs{y_1-y_2}\right). \end{equation}

Finally, we assume that \begin{equation}\label{initialbuc} u_0\in\BUC(\mathbb{R}^d).\end{equation}

Among the above, (\ref{coercive}), (\ref{bounded}), (\ref{hamcon}) and (\ref{initialbuc}) are necessary for the well-posedness of (\ref{intro_eq}), see Crandall, Ishii and Lions \cite{CIL}. The rest, i.e., (\ref{transgroup}), (\ref{statergodic}) and (\ref{assumption_stationary}), are necessary for the homogenization.  We remark that the continuity requirement, (\ref{hamcon}), may be relaxed to uniform continuity, locally in the gradient variable, at the cost of less explicit continuity estimates for the effective Hamiltonian to follow.

Throughout the paper we will assume each of the statements (\ref{transgroup})-(\ref{initialbuc}).  To avoid repeating all of these, we introduce a steady assumption. \begin{equation}\label{steady} \textrm{Assume}\:(\ref{transgroup}),\: (\ref{statergodic}),\: (\ref{assumption_stationary}),\: (\ref{coercive}),\:(\ref{hamcon}),\: (\ref{hamcon}),\: (\ref{initialbuc}). \end{equation}

\section{The Approximate Macroscopic Problem}

In this section, we begin our identification of the effective Hamiltonian by introducing, for each $p\in\mathbb{R}^d$, $\delta>0$ and $\omega\in\Omega$, the approximate macroscopic problem \begin{equation}\label{macroscopic} \delta v^\delta+H(p+Dv^\delta,y,\omega)=0\;\;\textrm{on}\;\;\mathbb{R}^d.\end{equation}  The following proposition outlines the most basic existence and regularity results for (\ref{macroscopic}).  Note that, when addressing the solutions $v^\delta$, we will typically suppress the dependence on $\omega\in\Omega$.  We emphasize, however, that for each $\delta>0$ and $p\in\mathbb{R}^d$, the family of solutions $\left\{v^\delta(\cdot,\omega)\right\}_{\omega\in\Omega}$ constitutes a random process $v^\delta:\mathbb{R}^d\times\Omega\rightarrow\mathbb{R}$.

\begin{prop}\label{macroscopic_posed}  Assume (\ref{steady}).  For each fixed $p\in\mathbb{R}^d$, $\delta>0$ and $\omega\in\Omega$, equation (\ref{macroscopic}) has a unique solution $v^\delta:\mathbb{R}^d\rightarrow\mathbb{R}$ satisfying, for $C=C(p)>0$, $$\norm{\delta v^\delta}_{L^\infty(\mathbb{R}^d)}\leq C\;\;\textrm{and}\;\;\norm{Dv^\delta}_{L^\infty(\mathbb{R}^d)}\leq C.$$  Furthermore, the processes $$(y,\omega)\rightarrow v^\delta(y,\omega)\;\;\textrm{and}\;\;(y,\omega)\rightarrow Dv^\delta(y,\omega)$$ are stationary in the sense of (\ref{statergodic}) with $$\mathbb{E}\left(Dv^\delta(0,\omega)\right)=0.$$\end{prop}

\begin{proof}  Fix $p\in\mathbb{R}^d$, $\delta>0$ and $\omega\in\Omega$.  In view of (\ref{bounded}) there exists $C_1=C_1(p)>0$ such that the constant functions $C_1/\delta$ and $-C_1/\delta$ are respectively a subsolution and supersolution of (\ref{macroscopic}).  The existence of a solution $v^\delta$ satisfying \begin{equation}\label{macroscopic_posed_1}-C_1/\delta\leq v^\delta \leq C_1/\delta\end{equation} follows by Perron's method.  Uniqueness is guaranteed by the usual comparison result, see \cite{CIL}.

The previous estimate combined with the coercivity (\ref{coercive}) implies the gradient bound.  In view of (\ref{coercive}) and (\ref{macroscopic_posed_1}), there exists $C_2=C_2(p)>0$ such that $$\norm{Dv^\delta(\cdot,\omega)}_{L^\infty(\mathbb{R}^d)}\leq C_2.$$  The estimate's proof is completed by taking $C=\max\left\{C_1,C_2\right\}$.

The stationarity follows from the uniqueness.  In view of (\ref{assumption_stationary}), for each $y\in\mathbb{R}^d$, the function $v^\delta(\cdot+y,\omega)$ satisfies (\ref{macroscopic}) corresponding to $p\in\mathbb{R}^d$ and $\tau_y\omega\in\Omega$.  Similarly, by definition, the function $v^\delta(\cdot,\tau_y\omega)$ satisfies (\ref{macroscopic}) corresponding to $p\in\mathbb{R}^d$ and $\tau_y\omega\in\Omega$.  Therefore, by uniqueness, for all $x,y\in\mathbb{R}^d$ and $\omega\in\Omega$, $$v^\delta(x+y,\omega)=v^\delta(x,\tau_y\omega),$$ which comples the proof of stationarity.

The mean zero gradient is a consequence of the stationarity.  By Rademacher's theorem, for each $\omega\in\Omega$, $Dv^\delta(\cdot,\omega)$ exists in the sense of distributions.  Furthermore, by the stationarity of $v^\delta(y,\omega)$, the process $$(y,\omega)\rightarrow Dv^\delta(y,\omega)\;\;\textrm{is stationary,}$$  with $Dv^\delta\in L^\infty(\mathbb{R}^d\times\Omega;\mathbb{R}^d)$. 

The stationarity and Fubini's theorem imply that, for any compactly supported $\phi\in \C^\infty(\mathbb{R}^d)$ with integral one, $$\mathbb{E}\left(Dv^\delta(0,\omega)\right)=\mathbb{E}\left(Dv^\delta(0,\omega)\right)\int_{\mathbb{R}^d}\phi(x)\;dx=\mathbb{E}\left(-v^\delta(0,\omega)\right)\int_{\mathbb{R}^d}D\phi(x)\;dx=0,$$ thereby competing the proof of the proposition.  \end{proof}

We conclude this section with a continuity estimate describing the dependence of the solutions $v^\delta$ on the gradient variable $p\in\mathbb{R}^d$ determining each instance of (\ref{macroscopic}).  The estimate follows by a simple comparison argument.

\begin{prop}\label{macroscopic_continuity}  Assume (\ref{steady}).  For each $R>0$ there exists $C=C(R)>0$ such that, for all $\omega\in\Omega$, for all $p_1,p_2\in B_R$, for $v^\delta_1,v^\delta_2$ the solutions of (\ref{macroscopic}) corresponding to $p_1,p_2$, $$\norm{\delta v^\delta_1-\delta v^\delta_2}_{L^\infty(\mathbb{R}^d)}\leq C\abs{p_1-p_2}.$$\end{prop}

\begin{proof}  Fix $R>0$, $\omega\in\Omega$ and $p_1,p_2\in B_R$.  We write $v^\delta_1$ and $v^\delta_2$ respectively for the solutions of (\ref{macroscopic}) corresponding to $p_1$ and $p_2$.

In view of Proposition \ref{macroscopic_posed} and (\ref{hamcon}) there exists $C=C(R)>0$ such that $$\tilde{v}^\delta_1=v^\delta_1-(C/\delta)\abs{p_1-p_2}$$ is a subsolution of (\ref{macroscopic}) corresponding to $p_2$.  This, with the standard comparison result, see \cite{CIL}, implies that $$\delta v^\delta_1-\delta v^\delta_2\leq C\abs{p_1-p_2}\;\;\textrm{on}\;\;\mathbb{R}^d.$$  The claim then follows by reversing the roles of $v^\delta_1$ and $v^\delta_2$, since $R>0$, $p_1$ and $p_2$ were arbitrary.  \end{proof}

\section{The Effective Hamiltonian}

We now present the definition of the effective Hamiltonian $\overline{H}:\mathbb{R}^d\rightarrow\mathbb{R}$ and characterize some of its properties.  In particular, we prove the well-posedness of the effective equation  \begin{equation}\label{effective_eq}\left\{\begin{array}{ll} \overline{u}_t+\overline{H}(D\overline{u})=0 & \textrm{on}\;\;\mathbb{R}^d\times(0,\infty), \\ \overline{u}=u_0 & \textrm{on}\;\;\mathbb{R}^d\times\left\{0\right\}.\end{array}\right.\end{equation}

Define, for each $p\in\mathbb{R}^d$ and $\omega\in\Omega$, for $v^\delta$ the solution of (\ref{macroscopic}) corresponding to $p\in\mathbb{R}^d$, $$\overline{H}(p,\omega)=\limsup_{\delta\rightarrow 0}-\delta v^\delta(0,\omega).$$  Our first remark is that, on a subset of full probability, for each $p\in\mathbb{R}^d$, $\overline{H}(p,\omega)$ is deterministic.

\begin{prop}\label{effective_deterministic}  Assume (\ref{steady}).  There exists a subset of full probability $\Omega_1\subset\Omega$ such that, for each $p\in\mathbb{R}^d$ there exists $\overline{H}(p)\in\mathbb{R}$ satisfying, for each $\omega\in\Omega_1$, $$\overline{H}(p,\omega)=\overline{H}(p).$$\end{prop}

\begin{proof}  Fix $p\in\mathbb{R}^d$, $\alpha\in\mathbb{R}$ and consider the event $$A_\alpha=\left\{\;\omega\in\Omega\;|\;H(p,\omega)>\alpha\;\right\}.$$  We will prove that this event is invariant under the transformation group $\left\{\tau_y\right\}_{y\in\mathbb{R}^d}.$

If $A_\alpha=\emptyset$ there is nothing to prove.  Otherwise, fix $y\in\mathbb{R}^d$, and let $\omega\in A_\alpha$.  In view of Proposition \ref{macroscopic_posed}, for each $\delta>0$, for $C>0$, $$\abs{\delta v^\delta(0,\tau_y \omega)-\delta v^\delta(0,\omega)}=\abs{\delta v^\delta(y,\omega)-\delta v^\delta(0,\omega)}\leq C\delta\abs{y}.$$  This implies that $\overline{H}(p,\omega)=\overline{H}(p,\tau_y\omega)$ and, therefore, that $\tau_y \left(A_\alpha\right)\subset A_\alpha$.  Since the same logic applies to $-y\in\mathbb{R}^d$, we conclude that $\tau_{-y}\left(A_\alpha\right)\subset A_\alpha$ and, therefore, that $A_\alpha\subset \tau_y\left(A_\alpha\right)$.

Since $\alpha\in\mathbb{R}$ and $y\in\mathbb{R}^d$ were arbitrary, (\ref{transgroup}) implies that, for each $\alpha\in\mathbb{R}$, $$\mathbb{P}(A_\alpha)=0\;\;\textrm{or}\;\;\mathbb{P}(A_\alpha)=1.$$  Therefore, since the $\delta v^\delta(0,\omega)$ were bounded in Proposition \ref{macroscopic_posed}, there exists a subset of full probability $\Omega_1(p)\subset\Omega$ such that, for each $\omega\in\Omega_1(p)$, $$\overline{H}(p,\omega)=\overline{H}(p):=\sup\left\{\;\alpha\in\mathbb{R}\;|\;\mathbb{P}(A_\alpha)=1\right\}.$$  To conclude, we define the subset of full probability $$\Omega_1=\bigcap_{p\in\mathbb{Q}^d}\Omega_1(p)$$ and apply Proposition \ref{macroscopic_continuity}.  \end{proof}

We therefore define, for each $p\in\mathbb{R}^d$, for any $\omega\in\Omega_1$, for $v^\delta$ the solution of (\ref{macroscopic}) corresponding to $p\in\mathbb{R}^d$, \begin{equation}\label{effective_def}\overline{H}(p)=\limsup_{\delta\rightarrow 0}-\delta v^\delta(0,\omega).\end{equation}  The following proposition proves that the effective Hamiltonian inherits the continuity and coercivity of the original Hamiltonian, as described in Section 2.

\begin{prop}\label{effective_properties}  Assume (\ref{steady}).  The effective Hamiltonian is locally Lipschitz continuous.  Precisely, for each $R>0$, there exists $C=C(R)>0$ such that, for every $p_1,p_2\in B_R$, $$\abs{\overline{H}(p_1)-\overline{H}(p_2)}\leq C\abs{p_1-p_2}.$$  The effective Hamiltonian is coercive, $$\lim_{\abs{p}\rightarrow\infty}\overline{H}(p)=\infty.$$  \end{prop}

\begin{proof}  The local Lipschitz continuity follows from Proposition \ref{macroscopic_continuity} and (\ref{effective_def}).

The coercivity follows by a similar comparison argument.  For each $\delta>0$, $\omega\in\Omega$ and $p\in\mathbb{R}^d$, the constant function $-\inf_{(y,\omega)\in\mathbb{R}^d\times\Omega}H(p,y,\omega)/\delta$ is a supersolution of (\ref{macroscopic}) corresponding to $p\in\mathbb{R}^d$.  Therefore, for each $\delta>0$, for $v^\delta$ the solution of (\ref{macroscopic}) corresponding to $p\in\mathbb{R}^d$, the comparison principle implies $$\inf_{(y,\omega)\in\mathbb{R}^d\times\Omega}H(p,y,\omega)\leq -\delta v^\delta(0,\omega).$$  The conclusion follows from (\ref{coercive}) and (\ref{effective_def}).   \end{proof}

The properties listed in Proposition \ref{effective_properties} are more than sufficient to guarantee the well-posedness of the homogenized problem.  This is a standard fact from the theory of viscosity solutions, see \cite{CIL}.  The gradient estimate follows as in Proposition \ref{macroscopic_posed} due to the coercivity seen in Proposition \ref{effective_properties}.

\begin{prop}\label{effective_posed}  Assume (\ref{steady}).  For each $u_0\in\BUC(\mathbb{R}^d)$, there exists a unique solution $\overline{u}:\mathbb{R}^d\times[0,\infty)\rightarrow\mathbb{R}$ of (\ref{effective_eq}) satisfying $\overline{u}\in \BUC([0,T)\times\mathbb{R}^d)$ for each $T>0$.  Furthermore, if $u_0\in\Lip(\mathbb{R}^d)$, then for $C=C(\norm{Du_0}_{L^\infty(\mathbb{R}^d)})>0$, $$\norm{\overline{u}_t}_{L^\infty(\mathbb{R}^d)}\leq C\;\;\textrm{and}\;\;\norm{D\overline{u}}_{L^\infty(\mathbb{R}^d)}\leq C.$$\end{prop}

We conclude this section by upgrading our current characterization of the effective Hamiltonian.  The following argument is motivated by the methods of \cite{LS2} and \cite{AS}.  However, notice here we obtain a global characterization of the effective Hamiltonian, a fact which will be used in the Section 5.

\begin{prop}\label{effective_ball} Assume (\ref{steady}).  There exists subset $\Omega_2\subset\Omega_1$ of full probability such that, for each $\omega\in\Omega_2$, $p\in\mathbb{R}^d$, $y\in\mathbb{R}^d$ and $R>0$, for $v^\delta$ the solution of (\ref{macroscopic}) corresponding to $p\in\mathbb{R}^d$, $$\overline{H}(p)=\limsup_{\delta\rightarrow 0}-\delta v^\delta(y/\delta,\omega)=\limsup_{\delta\rightarrow 0}\sup_{x\in B_{R/\delta}(y/\delta)}-\delta v^\delta(x,\omega).$$\end{prop}

\begin{proof}  Fix $p\in\mathbb{R}^d$ and $y_0\in\mathbb{R}^d$.  We write $v^\delta$ for the solution of (\ref{macroscopic}) corresponding to $p$ and define, for each $\omega\in\Omega$, $$\tilde{H}(p,y_0,\omega):=\limsup_{\delta\rightarrow 0} -\delta v^\delta(y_0/\delta,\omega).$$  In view of Proposition \ref{macroscopic_posed} and (\ref{transgroup}), following a repetition of the proof appearing in Proposition \ref{effective_deterministic}, there exists a subset $A_1=A_1(y_0,p)\subset\Omega$ and $\tilde{H}(p,y_0)\in\mathbb{R}$ such that, for all $\omega\in A_1$, \begin{equation}\label{effective_ball_1} \limsup_{\delta\rightarrow 0}-\delta v^\delta(y_0/\delta,\omega)=\tilde{H}(p,y_0).\end{equation}

Fix $0<\rho<1$.  Egorov's theorem implies that there exists a subset $E_\rho\subset\Omega$ and $\overline{\delta}_1=\overline{\delta}_1(\rho)>0$ such that, for all $0<\delta<\overline{\delta}_1$, \begin{equation}\label{effective_ball_2}\sup_{\omega\in E_\rho}-\delta v^\delta(y_0/\delta,\omega)-\tilde{H}(p,y_0)<\rho\;\;\textrm{with}\;\;\mathbb{P}(E_\rho)>1-\rho.\end{equation}  In what follows, we write $\chi_{E_\rho}$ to denote the indicator function of the set $E_\rho\subset\Omega$.

The ergodic theorem, see Theorem \ref{ergodic1}, (\ref{transgroup}) and (\ref{effective_ball_2}) imply that there exists a subset $A_2=A_2(p,y_0)\subset \Omega$ of full probability such that, for each $\omega\in A_2$ and $R>0$, \begin{equation}\label{effective_ball_3} \lim_{\delta\rightarrow 0}\dashint_{B_{R/\delta}}\chi_{E_\rho}(\tau_y\omega)\;dy=\mathbb{E}\left(\chi_{E_\rho}\right)>1-\rho.\end{equation}  Fix $R>0$ and, using (\ref{effective_ball_3}), for each $\omega\in A_2$ choose $\overline{\delta}_2=\overline{\delta}_2(R,\rho,\omega)>0$ such that, for all $0<\delta<\overline{\delta}_2$,  \begin{equation}\label{effective_ball_4} \int_{B_{R/\delta}}\chi_{E_\rho}(\tau_y\omega)\;dy>\abs{B_{R/\delta}}(1-\rho),\end{equation} and, for each $\omega\in A_2$, define $\overline{\delta}=\overline{\delta}(R,\rho,\omega)=\min\left\{\overline{\delta}_1,\overline{\delta}_2\right\}$.

Fix $\omega\in A_2$.  In view of (\ref{effective_ball_4}), for every $0<\delta<\overline{\delta}$, whenever $y\in B_{R/\delta}$ there exists $z\in B_{R/\delta}$ satisfying, for $C=C(R)>0$,  \begin{equation}\label{effective_ball_5} \abs{y-z}\leq C\delta^{-1}\rho^{\frac{1}{d}}R\;\;\textrm{with}\;\;\tau_z\omega\in E_\rho. \end{equation}  Notice that Proposition \ref{macroscopic_posed}, $0<\delta<\overline{\delta}$ and $\tau_z\in E_\rho$ imply \begin{equation}\label{effective_ball_6} -\delta v^\delta(y_0/\delta,\tau_z\omega)-\tilde{H}(p,y_0)=-\delta v^\delta(y_0/\delta+z,\omega)-\tilde{H}(p,y_0)<\rho.\end{equation}  Fix $y_1\in B_{R/\delta}$.  There exists $z_1\in B_{R/\delta}$ satisfying (\ref{effective_ball_5}) and (\ref{effective_ball_6}), which implies that \begin{multline*} -\delta v^\delta(y_0/\delta+y_1,\omega)-\tilde{H}(p,y_0) \leq \abs{-\delta v^\delta(y_0/\delta+y_1,\omega)+\delta v^\delta(y_0/\delta+z_1,\omega)} \\ -\delta v^\delta(y_0/\delta+z_1,\omega)-\tilde{H}(p,y_0),\end{multline*} and, therefore, using (\ref{effective_ball_5}), (\ref{effective_ball_6}) and Proposition \ref{macroscopic_posed}, for $C>0$ independent of $0<\delta<\overline{\delta}$,  \begin{equation}\label{effective_ball_7} -\delta v^\delta(y_0/\delta+y_1,\omega)-\tilde{H}(p,y_0) \leq C\rho^{\frac{1}{d}}R+\rho.\end{equation}

Since $R>0$, $\omega\in A_2$, $y_1\in B_{R/\delta}$ and $0<\rho<1$ were arbitrary, we conclude that, for every $\omega\in A_2$ and $R>0$, $$\limsup_{\delta\rightarrow 0} \sup_{x\in B_{R/\delta}(y_0/\delta)}-\delta v^\delta(x/\delta,\omega)\leq \tilde{H}(p,y_0).$$  We therefore define the subset of full probability $A_3=A_3(p,y_0)=A_1\cap A_2$ and conclude that, in view of (\ref{effective_ball_1}), for each $\omega\in A_3$ and $R>0$, \begin{equation}\label{effective_ball_8} \tilde{H}(p,y_0)=\limsup_{\delta\rightarrow 0}-\delta v^\delta(y_0/\delta,\omega)=\limsup_{\delta\rightarrow 0}\sup_{x\in B_{R/\delta}(y_0/\delta)}-\delta v^\delta(x,\omega).\end{equation}  Finally, define $$\Omega_2=\bigcap_{(p,y)\in\mathbb{Q}^d\times\mathbb{Q}^d}A_3(p,y),$$ and conclude, using Proposition \ref{macroscopic_posed} and Proposition \ref{macroscopic_continuity}, that for every $\omega\in\Omega_2$, $y\in\mathbb{R}^d$, $p\in\mathbb{R}^d$ and $R>0$, for $v^\delta$ the solution of (\ref{macroscopic}) corresponding to $p$, \begin{equation}\label{effective_ball_9}\tilde{H}(p,y)=\limsup_{\delta\rightarrow 0}-\delta v^\delta(y/\delta,\omega)=\limsup_{\delta\rightarrow 0}\sup_{x\in B_{R/\delta}(y/\delta)}-\delta v^\delta(x,\omega).\end{equation}  It remains to prove that, for each $y\in\mathbb{R}^d$ and $p\in\mathbb{R}^d$, $\tilde{H}(p,y)=\overline{H}(p).$

Fix $p\in\mathbb{R}^d$ and let $y_1,y_2\in\mathbb{R}^d$ with $R>\abs{y_1-y_2}$.  Then, since $y_2\in B_R(y_1)$, (\ref{effective_ball_9}) implies, for each $\omega\in\Omega_2$, $$\tilde{H}(p,y_2)\leq \limsup_{\delta\rightarrow 0}\sup_{x\in B_{R/\delta}(y/\delta)}-\delta v^\delta(x,\omega)=\tilde{H}(p,y_1).$$  Since the opposite inequality follows by an identical argument, we have $\tilde{H}(p,y_1)=\tilde{H}(p,y_2)$.  In particular, by definition, for all $y\in\mathbb{R}^d$, $$\tilde{H}(p,y)=\tilde{H}(p,0)=\overline{H}(p).$$  Since $p\in\mathbb{R}^d$ was arbitrary, this completes the argument.  \end{proof}

\section{The Metric Problem}

We now analyze, in an informal sense, the metric properties of the Hamiltonian $H(p,y,\omega)$ appearing in (\ref{intro_eq}).  Namely, for each $p\in\mathbb{R}^d$, we construct a process $m:\mathbb{R}^d\times\mathbb{R}^d\times\Omega\rightarrow\mathbb{R}$ such that, on a subset of full probability, for each $y\in\mathbb{R}^d$, the function $m(\cdot,y,\omega)$ satisfies $m(y,y,\omega)=0$ with \begin{equation}\label{metric_eq} H(p+Dm,x,\omega)\leq\overline{H}(p)\;\;\textrm{on}\;\;\mathbb{R}^d\;\;\textrm{and}\;\; H(p+Dm,x,\omega)=\overline{H}(p)\;\;\textrm{on}\;\;\mathbb{R}^d\setminus\left\{y\right\}.  \end{equation}  Furthemore, $m(x,y,\omega)$ is sub-additive, for each $x,y,z\in\mathbb{R}^d$ and $\omega\in\Omega$, \begin{equation}\label{metric_sub} m(x,y,\omega)\leq m(x,z,\omega)+m(z,y,\omega),\end{equation} jointly-stationary, for each $x,y,z\in\mathbb{R}^d$ and $\omega\in\Omega$, \begin{equation}\label{metric_joint}m(x,y,\tau_z,\omega)=m(x+z,y+z,\omega),\end{equation} and, for each $x,y\in\mathbb{R}^d$ and $\omega\in\Omega$, there exists $C>0$ independent of $x,y\in\mathbb{R}^d$ and $\omega\in\Omega$ satisfying \begin{equation}\label{metric_bound} \abs{m(x,y,\omega)}\leq C\abs{x-y}.\end{equation}  The interpretation is that $m(x,y,\omega)$ represents the cost of traveling from $x$ to $y$ in the environment described by $H(p+\cdot,\cdot,\omega):\mathbb{R}^d\times\mathbb{R}^d\rightarrow\mathbb{R}$, see Lions \cite{L} for a formalization of this intuition.  We remark that the conditions (\ref{metric_sub}), (\ref{metric_joint}) and (\ref{metric_bound}) will allow us to apply the sub-additive ergodic theorem below, see Theorem \ref{subadditiveergodic}.

\begin{prop}\label{metric_posed} Assume (\ref{steady}).  For every $p\in\mathbb{R}^d$, there exists a process $m(x,y,\omega):\mathbb{R}^d\times\mathbb{R}^d\times\Omega\rightarrow\mathbb{R}$ satisfying (\ref{metric_eq}), (\ref{metric_sub}), (\ref{metric_joint}) and (\ref{metric_bound}) for every $\omega\in\Omega_2$. \end{prop}

\begin{proof}  Fix $p\in\mathbb{R}^d$ and write $v^\delta$ for the solution of (\ref{macroscopic}) corresponding to $p$.  We define, for each $\delta>0$, $$w^\delta(x,y,\omega):=v^\delta(x,\omega)-v^\delta(y,\omega),$$  and, \begin{equation}\label{metric_posed_9}w^*(x,y,\omega):=\limsup_{\delta\rightarrow 0}w^\delta(x,y,\omega).\end{equation}

Observe that, since Proposition \ref{macroscopic_posed} implies that the family $\left\{w^\delta(x,y,\omega)\right\}_{\delta>0}$ are uniformly equicontinuous in $x,y\in\mathbb{R}^d$, for each $\omega\in\Omega$, standard properties of viscosity solutions, see \cite{CIL}, imply that for every $\omega\in\Omega$ and $x,y\in\mathbb{R}^d$, the function $w^*(\cdot,y,\omega)$ satisfies \begin{equation}\label{metric_posed_1} H(p+Dw^*(x,y,\omega),x,\omega)\leq \limsup_{\delta\rightarrow 0}-\delta v^\delta(y,\omega)\;\;\textrm{on}\;\;\mathbb{R}^d.\end{equation}  In particular, for every $\omega\in\Omega_2$, using Proposition \ref{effective_ball},\begin{equation}\label{metric_posed_2} H(p+Dw^*(x,y,\omega),x,\omega)\leq \overline{H}(p)\;\;\textrm{on}\;\;\mathbb{R}^d.\end{equation}

In view of Proposition \ref{macroscopic_posed} and (\ref{coercive}), there exists $C_1>0$ such that, for every $\omega\in\Omega$, the map \begin{equation}\label{metric_posed_3} x \rightarrow C_1\abs{x}\;\;\textrm{is a supersolution of (\ref{metric_posed_1}) on}\;\;\mathbb{R}^d\setminus\left\{0\right\}.\end{equation}  Define, for each $\omega\in\Omega$ and $x,y\in\mathbb{R}^d$, \begin{multline}\label{metric_posed_4}\mathcal{A}(x,y,\omega)=\left\{\;z(x)-z(y)\;| z(\cdot)\in\C(\mathbb{R}^d),\;\;z(x)-z(y)\leq C_1\abs{x-y}\right. \\ \left.\textrm{and}\;\;z\;\;\textrm{satisfies}\;(\ref{metric_posed_1})\;\textrm{for this}\; \omega\in\Omega.\;\right\}.\end{multline}  Notice that, for each $x,y\in\mathbb{R}^d$ and $\omega\in\Omega$, $\mathcal{A}(x,y,\omega)\neq \emptyset$ since $w^*(x,y,\omega)\in\mathcal{A}(x,y,\omega).$

We remark that, for first-order equations, the assumption $z(x)-z(y)\leq C_1\abs{x-y}$ appearing in the definition of $\mathcal{A}(x,y,\omega)$ is redundant by the choice of $C_1>0$, since subsolutions are necessarily Lipschitz continuous with Lipschitz constant less than or equal to $C_1$.  However, for second-order equations, a similar such assumption is necessary to apply Perron's method below, since we have no comparison principle.  See Section 8 for an outline of the details.

We define, for each $\omega\in\Omega$ and $x,y\in\mathbb{R}^d$, \begin{equation}\label{metric_posed_5}m(x,y,\omega)=\sup\mathcal{A}(x,y,\omega).\end{equation}  It follows by definition that, for each $x,y,z,\in\mathbb{R}^d$ and $\omega\in\Omega$, \begin{equation}\label{metric_posed_6} \sup\mathcal{A}(x,y,\omega)\leq\sup\mathcal{A}(x,z,\omega)+\sup\mathcal{A}(z,y,\omega)\;\;\textrm{and}\;\;\mathcal{A}(x,y,\tau_z\omega)=\mathcal{A}(x+z,y+z,\omega).\end{equation}  Therefore, $m(x,y,\omega)$ satisfies (\ref{metric_sub}) and (\ref{metric_joint}).  Furthermore, again by the definition of $\mathcal{A}(x,y,\omega)$, for $C_1>0$ defined above, for each $\omega\in\Omega$ and $x,y\in\mathbb{R}^d$, \begin{equation}\label{metric_posed_7} w^*(x,y,\omega)\leq m(x,y,\omega)\leq C_1\abs{x-y}\;\;\textrm{and, therefore,}\;\;\abs{m(x,y,\omega)}\leq C_1\abs{x-y}.\end{equation}  Finally, in view of (\ref{metric_posed_3}) and the definition, Perron's method, see \cite{CIL}, implies that, for every $x,y\in\mathbb{R}^d$ and $\omega\in\Omega$, $m(\cdot,y,\omega)$ satisfies $m(y,y,\omega)=0$ with $$H(p+Dm,x,\omega)\leq\limsup_{\delta\rightarrow 0}-\delta v^\delta(y,\omega)\;\;\textrm{on}\;\;\mathbb{R}^d,$$ and, $$H(p+Dm,x,\omega)=\limsup_{\delta\rightarrow 0}-\delta v^\delta(y,\omega)\;\;\textrm{on}\;\;\mathbb{R}^d\setminus\left\{y\right\}.$$  In particular, for every $\omega\in\Omega_2$, using Proposition \ref{effective_ball}, \begin{equation}\label{metric_posed_8}H(p+Dm,x,\omega)\leq\overline{H}(p)\;\;\textrm{on}\;\;\mathbb{R}^d\;\;\textrm{and}\;\; H(p+Dm,x,\omega)=\overline{H}(p)\;\;\textrm{on}\;\;\mathbb{R}^d\setminus\left\{y\right\}.\end{equation}  In view of (\ref{metric_posed_6}), (\ref{metric_posed_7}) and (\ref{metric_posed_8}) the proof is complete.\end{proof}

Henceforth, for each $p\in\mathbb{R}^d$, we write $m(x,y,\omega)$ for the solution of the (\ref{metric_eq}) constructed in Proposition \ref{metric_posed}.  In particular, we remark that, for each $x,y\in\mathbb{R}^d$ and $\omega\in\Omega$, for $w^*(x,y,\omega)$ defined for this $p\in\mathbb{R}^d$ in (\ref{metric_posed_9}), \begin{equation}\label{metric_lower}w^*(x,y,\omega)\leq m(x,y,\omega).\end{equation}

Our aim is to infer, for each $p\in\mathbb{R}^d$, the metric properties of $\overline{H}(p+\cdot):\mathbb{R}^d\rightarrow\mathbb{R}$ from the asymptotic behavior of the solutions $m(x,y,\omega)$.  We will not succeed, in full, but toward this goal we are able we identify, for each $p\in\mathbb{R}^d$, $\overline{m}\in\Lip(\mathbb{R}^d)$ such that there exists a subset of full probability satisfying, for each $x,y\in\mathbb{R}^d$, as $\epsilon\rightarrow 0$, for $m^\epsilon$ the solution of (\ref{metric_eq}) corresponding to $p\in\mathbb{R}^d$, $$\epsilon m^\epsilon(x/\epsilon,y/\epsilon,\omega)\rightarrow \overline{m}(x-y)\;\;\textrm{locally uniformly on}\;\;\mathbb{R}^d\;\;\textrm{with}\;\;\overline{m}\geq 0\;\;\textrm{on}\;\;\mathbb{R}^d.$$  Furthermore we conclude, using Proposition \ref{effective_ball}, $$\overline{H}(p+D\overline{m})\leq\overline{H}(p)\;\;\textrm{on}\;\;\mathbb{R}^d.$$  This turns out to be, for our purposes, a sufficient characterization of the metric properties of the effective Hamiltonian to complete the argument, as described in Proposition \ref{metric_ray} and Section 6.

\begin{prop}\label{metric_limit}  Assume (\ref{steady}).  For each $p\in\mathbb{R}^d$, there exists $\overline{m}\in\Lip(\mathbb{R}^d)$ and a subset $\Omega_3=\Omega_3(p)\subset\Omega$ of full probability such that, as $\epsilon\rightarrow 0$, for each $\omega\in\Omega_3$ and $x,y\in\mathbb{R}^d$, for $m(x,y,\omega)$ defined in Proposition \ref{metric_posed}, $$\lim_{\epsilon\rightarrow 0}\epsilon m(x/\epsilon,y/\epsilon,\omega)\rightarrow \overline{m}(x-y)\;\;\textrm{locally uniformly on}\;\;\mathbb{R}^d\;\;\textrm{with}\;\;\overline{m}\geq 0\;\;\textrm{on}\;\;\mathbb{R}^d.$$\end{prop}

\begin{proof}  Fix $p\in\mathbb{R}^d$.  Proposition \ref{metric_posed} and the sub-additive ergodic theorem, see Theorem \ref{subadditiveergodic}, imply that there exists a subset of full probability $A_1=A_1(p)\subset\Omega$ such that, for each $\omega\in A_1$, for each $x,y\in\mathbb{R}^d$, as $\epsilon\rightarrow 0$, for $\overline{m}(x-y,\omega)\in L^\infty(\Omega)$ and $m^\epsilon$ the solution of (\ref{metric_eq}) corresponding to $p$, \begin{equation}\label{metric_limit_1}\epsilon m(x/\epsilon,y/\epsilon,\omega)\rightarrow \overline{m}(x-y,\omega).\end{equation}

Since, for each $\omega\in\Omega$, $m(\cdot,0,\omega)\in\Lip(\mathbb{R}^d)$, the convergence (\ref{metric_limit_1}), Proposition \ref{metric_posed} and a repetition of the proof appearing in Proposition \ref{effective_deterministic} proves that there exists a subset $\Omega_3=\Omega_3(p)\subset A_1(p)$ of full probability such that, for each $\omega\in \Omega_3$, for each $y\in\mathbb{R}^d$ there exists $\overline{m}(y)\in\mathbb{R}$ satisfying \begin{equation}\label{metric_limit_2}\overline{m}(y,\omega)=\overline{m}(y)\;\;\textrm{with}\;\;\overline{m}(0)=0\;\;\textrm{and}\;\;\overline{m}\in\Lip(\mathbb{R}^d).\end{equation}  We conclude that, for each $\omega\in\Omega_3$, for all $x,y\in\mathbb{R}^d$, for $m^\epsilon$ the solution of (\ref{metric_eq}) corresponding to $p\in\mathbb{R}^d$, $$\lim_{\epsilon\rightarrow 0}\epsilon m^\epsilon(x/\epsilon,y/\epsilon,\omega)=\overline{m}(x-y).$$  It remains to prove that $\overline{m}\geq 0$ on $\mathbb{R}^d$.

Define, for each $\delta>0$, for $v^\delta$ the solutions of (\ref{macroscopic}) corresponding to $p\in\mathbb{R}^d$, $$w^\delta(y,\omega)=v^\delta(y,\omega)-v^\delta(0,\omega).$$  Notice that, in view of Proposition \ref{macroscopic_posed}, for every $\omega\in\Omega$, the family $\left\{w^\delta(\cdot,\omega)\right\}_{\delta>0}$ is locally bounded and Lipschitz continuous, uniformly in $\delta>0$, and satisfies $Dw^\delta(y,\omega)=Dv^\delta(y,\omega)$ in the sense of distributions.  Therefore, the process \begin{equation}\label{effective_strong_2} (y,\omega)\rightarrow Dw^\delta(y,\omega)\;\;\textrm{is stationary.}\end{equation}  We conclude that there exists $w\in L^{d+1}(\mathbb{R}^d\times\Omega)$, $W\in L^{d+1}(\mathbb{R}^d\times\Omega;\mathbb{R}^d)$ and a subsequence $\left\{\delta_k\rightarrow 0\right\}_{k=1}^\infty$ such that, for each $R>0$, as $k\rightarrow\infty$, \begin{equation}\label{effective_strong_1}\begin{array}{cl} w^{\delta_k}\rightharpoonup w & \textrm{weakly in}\;\;L^{d+1}(B_R\times\Omega), \\ Dw^{\delta_k}\rightharpoonup W & \textrm{weakly in}\;\;L^{d+1}(B_R\times\Omega;\mathbb{R}^d).  \end{array}\end{equation}  By taking a countable sequence of radii $\left\{R_n\rightarrow\infty\right\}_{n=1}^\infty$, we conclude using Proposition \ref{appendix_distributional} of the appendix and Fubini's theorem that, on a subset $A_2=A_2(p)\subset\Omega_1$ of full probability, for every $\omega\in A_2$ and every $R>0$, $$w(\cdot,\omega)\in L^{d+1}(B_R)\;\;\textrm{and}\;\;W(\cdot,\omega)=Dw(\cdot,\omega)\;\;\textrm{in the sense of distributions in}\;\;L^{d+1}(B_R;\mathbb{R}^d).$$  Therefore, the Sobolev embedding theorem implies, for every $\omega\in A_2$, \begin{equation}\label{effective_strong_3} w(\cdot,\omega)\in\C(\mathbb{R}^d).\end{equation}

Indeed, since (\ref{effective_strong_1}) implies that, for each $R>0$, there exists, for each $1\leq k<\infty$, $z_k\in\textrm{Conv}(\left\{w_j\right\}_{j=k}^\infty)\subset L^{d+1}(B_R\times\Omega)$ satisfying, as $k\rightarrow\infty$, $$z_k\rightarrow w\;\;\textrm{strongly in}\;\;L^{d+1}(B_R\times\Omega),$$ Fubini's theorem, Proposition \ref{macroscopic_posed} and (\ref{effective_strong_3}) imply that there exists a subset $A_3=A_3(p)\subset A_2$ of full probability such that, for every $\omega\in A_3$, \begin{equation}\label{effective_strong_6} w(\cdot,\omega)\in \Lip(\mathbb{R}^d).\end{equation}

Furthermore, in view of Proposition \ref{macroscopic_posed}, (\ref{effective_strong_2}) and (\ref{effective_strong_1}), the process \begin{equation}\label{effective_strong_4} (y,\omega)\rightarrow Dw(y,\omega)\;\;\textrm{is stationary with}\;\;\mathbb{E}\left(Dw(0,\omega)\right)=0\;\;\textrm{and}\;\;\mathbb{E}\left(\abs{Dw(0,\omega)}\right)<\infty.\end{equation}  Therefore, in view of (\ref{effective_strong_6}) and Proposition \ref{appendix_sublinear} of the appendix, there exists a subset $A_4=A_4(p)\subset A_3$ of full probability such that, for every $\omega\in A_4$, \begin{equation}\label{effective_strong_5} \lim_{\abs{y}\rightarrow\infty}\frac{w(y,\omega)}{\abs{y}}=0.\end{equation}  Finally, (\ref{metric_posed_9}), (\ref{metric_lower}), (\ref{effective_strong_1}), (\ref{effective_strong_3}) and Fubini's theorem imply that there exist a subset $A_5=A_5(p)\subset A_4$ of full probability such that, for every $\omega\in A_5$, \begin{equation}\label{metric_limit_3} w(x,\omega)\leq w^*(x,0,\omega)\;\;\textrm{for all}\;\;x\in\mathbb{R}^d.\end{equation}

We are now prepared to conclude.  Fix $\omega\in\Omega_3\cap A_5$.  Then, in view of (\ref{metric_lower}), (\ref{effective_strong_5}) and (\ref{metric_limit_3}), for each $x\in\mathbb{R}^d$, $$0=\lim_{\epsilon\rightarrow 0}\epsilon w(x/\epsilon,\omega)\leq \liminf_{\epsilon\rightarrow 0}w^*(x/\epsilon,0,\omega)\leq \lim_{\epsilon\rightarrow 0}\epsilon m(x/\epsilon,0,\omega)=\overline{m}(x),$$ which, since $p\in\mathbb{R}^d$ was arbitrary, completes the argument.  \end{proof}

We now prove that, for each $p\in\mathbb{R}^d$, $\overline{m}$ provides a one-sided characterization of the metric properties of $\overline{H}(p+\cdot)$ in the sense that $$\overline{H}(p+D\overline{m})\leq \overline{H}(p)\;\;\textrm{on}\;\;\mathbb{R}^d.$$  The proof follows from Proposition \ref{effective_ball} and a small variation of the perturbed test function method.

\begin{prop}\label{metric_homogenization}  Assume (\ref{steady}).  For each $p\in\mathbb{R}^d$, for $\overline{m}$ constructed in Proposition \ref{metric_limit} corresponding to $p\in\mathbb{R}^d$, $$\overline{H}(p+D\overline{m})\leq\overline{H}(p)\;\;\textrm{on}\;\;\mathbb{R}^d.$$\end{prop}

\begin{proof}  Fix $p\in\mathbb{R}^d$ and write $\overline{m}$ for the function constructed in Proposition \ref{metric_limit} corresponding to $p$.  We proceed by contradiction.  Suppose that, for $x_0\in \mathbb{R}^d$ and $\phi\in \C^2(\mathbb{R}^d)$, \begin{equation}\label{metric_homogenization_1} \overline{m}-\phi\;\;\textrm{has a strict local maximum at}\;\;x_0\;\;\textrm{with}\;\;\overline{H}(p+D\phi(x_0))=\overline{H}(p)+\theta>\overline{H}(p).\end{equation}  Furthermore, fix $\overline{r}_1>0$ such that, for each $0<r<\overline{r}_1$, \begin{equation}\label{metric_homogenization_2}\sup_{x\in\partial B_r(x_0)}(u(x)-\phi(x))<\sup_{x\in B_r(x_0)}(u(x)-\phi(x))=u(x_0)-\phi(x_0).\end{equation}

For each $\delta>0$, write $v^\delta$ for the solution of (\ref{macroscopic}) corresponding to $p+D\phi(x_0)$.  In view of Proposition \ref{effective_ball}, there exists a subsequence $\left\{\delta_k\rightarrow 0\right\}_{k=1}^\infty$ such that, for each $\omega\in\Omega_2$, as $k\rightarrow\infty$, \begin{equation}\label{metric_homogenization_3}-\delta_kv^{\delta_k}(x_0/\delta,\omega)\rightarrow\overline{H}(p+D\phi(x_0)).\end{equation}  In view of Proposition \ref{macroscopic_posed}, fix $0<\overline{r}_2\leq\overline{r}_1$ such that, for each $\omega\in\Omega_2$ and $0<r<\overline{r}_2$, \begin{equation}\label{metric_homogenization_4}\liminf_{k\rightarrow\infty} \inf_{x\in B_{r/\delta_k}(x_0/\delta_k)}-\delta_k v^{\delta_k}(x,\omega)>\overline{H}(p)+\theta/2.\end{equation}

Fix $\omega\in\Omega_2\cap\Omega_3(p+D\phi(x_0))$.  Define for each $1\leq k<\infty$ and $\omega\in\Omega$, $$w^{\delta_k}(x,\omega)=v^{\delta_k}(x,\omega)-v^{\delta_k}(x_0/\delta_k,\omega),$$ and the perturbed test function $$\phi_k(x)=\phi+\delta_kw^{\delta_k}(x/\delta_k,\omega).$$  We will prove that there exists $0<\overline{r}<\overline{r}_2$ and $\overline{k}\geq 1$ such that, for all $k\geq\overline{k}$, $\phi_k$ satisfies $$H(p+D\phi_k,x,\omega)>\overline{H}(p)+\theta/3\;\;\textrm{on}\;\;B_{\overline{r}}(x_0).$$

Suppose that, for $y_0\in\mathbb{R}^d$ and $\eta\in\C^2(\mathbb{R}^d)$, $$\phi_k-\eta\;\;\textrm{has a local minimum at}\;\;y_0.$$  Then, the function $$w^{\delta_k}(x)-\frac{1}{\delta_k}(\eta(\delta_k x)-\phi(\delta_k x))\;\;\textrm{has a local minimum at}\;\; y_0/\delta _k.$$  This implies that \begin{equation}\label{metric_homogenization_5} H(p+D\phi(x_0)-D\phi(y_0)+D\eta(y_0),y_0/\delta_k,\omega)\geq -\delta_k v^{\delta_k}(y_0/\delta_k,\omega).\end{equation}

Proposition \ref{macroscopic_posed}, (\ref{hamcon}), $\phi\in \C^2(\mathbb{R}^d)$ and (\ref{metric_homogenization_4}) imply that there exists $\overline{k}\geq 1$ and $0<\overline{r}< \overline{r}_2$ such that, whenever $\abs{y_0-x_0}\leq \overline{r}$ and $k\geq \overline{k}$, \begin{equation}\label{metric_homogenization_6} H(p+D\eta(y_0),y_0/\delta_k,\omega)\geq \overline{H}(p)+\theta/3.  \end{equation}  Therefore, for each $k\geq \overline{k}$, $\phi_k$ is a strict supersolution of the rescaled metric problem corresponding to $p\in\mathbb{R}^d$ on $B_{\overline{r}}(x_0)$.

For each $\epsilon>0$, for $m(x,y,\omega)$ the solution constructed in Proposition \ref{metric_posed} corresponding to $p\in\mathbb{R}^d$, write $$m^\epsilon(x)=\epsilon m(x/\epsilon,0,\omega).$$  Notice that, in view of Proposition \ref{metric_posed}, $m^\epsilon$ satisfies $$H(p+Dm^\epsilon,x/\epsilon,\omega)\leq\overline{H}(p)\;\;\textrm{on}\;\;\mathbb{R}^d.$$  The comparison principle, see \cite{CIL}, implies that, for each $k\geq\overline{k}$, $$\sup_{x\in\partial B_{\overline{r}}(x_0)}(m^{\delta_k}(x)-\phi_k(x))=\sup_{x\in B_{\overline{r}}(x_0)}(m^{\delta_k}(x)-\phi_k(x)).$$  Observe that, after passing to a subsequence $\left\{\delta_{k_j}\rightarrow 0\right\}_{j=1}^\infty$ of $\left\{\delta_k\rightarrow0\right\}_{k=\overline{k}}^\infty$, Proposition \ref{macroscopic_posed} and Proposition \ref{effective_ball} imply that, for $\overline{w}\in\Lip(\mathbb{R}^d)$, as $j\rightarrow\infty$, $$w^{\delta_{k_j}}(x/\delta_{k_j})\rightarrow \overline{w}(x)\;\;\textrm{uniformly on}\;\;B_{\overline{r}}(x_0)\;\;\textrm{with}\;\;\overline{w}(x_0)=0\;\;\textrm{and}\;\;\overline{w}\geq 0\;\;\textrm{on}\;\;B_{\overline{r}}(x_0).$$  Therefore, since $\omega\in\Omega_2\cap\Omega_3(p+D\phi(x_0))$ and $0<\overline{r}<\overline{r}_1$, we conclude that \begin{multline*}\overline{m}(x_0)-\phi(x_0)=\lim_{j\rightarrow\infty}\sup_{x\in B_{\overline{r}}(x_0)}(m^{\delta_{k_j}}(x)-\phi_{k_j}(x)) \\ =\lim_{j\rightarrow\infty}\sup_{x\in \partial B_{\overline{r}}(x_0)}(m^{\delta_{k_j}}(x)-\phi_{k_j}(x))\leq \sup_{x\in\partial B_{\overline{r}}(x_0)}(\overline{m}(x)-\phi(x)),\end{multline*} which contradicts (\ref{metric_homogenization_2}) and completes the proof.  \end{proof}

In what follows, for every $p\in\mathbb{R}^d$, we write $$\left\{\;\overline{H}(q)\leq\overline{H}(p)\;\right\}:=\left\{\;q\in\mathbb{R}^d\;|\;\overline{H}(q)\leq\overline{H}(p)\;\right\}.$$  We now use Proposition \ref{metric_homogenization} to prove that, whenever $p\in\mathbb{R}^d$ satisfies \begin{equation}\label{metric_boundary} p\in\partial\left(\left\{\;\overline{H}(q)\leq\overline{H}(p)\;\right\}\right)\cap\partial\left(\textrm{Conv}\left(\left\{\;\overline{H}(q)\leq\overline{H}(p)\;\right\}\right)\right),\end{equation} there exists a ray of zero cost travel in the homogenized environment described by $\overline{H}(p+\cdot):\mathbb{R}^d\rightarrow\mathbb{R}.$  It will be along this ray that we obtain a control for the convergence of the $-\delta v^\delta$ in Section 6.

\begin{prop}\label{metric_ray}  Assume (\ref{steady}).  For every $p\in\mathbb{R}^d$ satisfying (\ref{metric_boundary}), for $\overline{m}$ constructed in Proposition \ref{metric_limit} corresponding to $p$, there exists $\nu=\nu(p)\in\mathbb{R}^d$ satisfying $\abs{\nu}=1$ such that, for all $t\geq 0$, $$\overline{m}(t\nu)=0.$$\end{prop}

\begin{proof}  Fix $p\in\mathbb{R}^d$ satisfying (\ref{metric_boundary}).  This implies that $0\in\mathbb{R}^d$ satisfies $$0\in\partial\left(\textrm{Conv}\left(\left\{\;\overline{H}(p+q)\leq\overline{H}(p)\;\right\}\right)\right)=\partial\left(\textrm{Conv}\left(\left\{\;q\in\mathbb{R}^d\;|\;\overline{H}(p+q)\leq\overline{H}(p)\;\right\}\right)\right).$$  The convexity therefore implies that there exists $\nu\in\mathbb{R}^d$ satisfying $\abs{\nu}=1$ with \begin{equation}\label{metric_ray_1}\nu\cdot x\leq 0\;\;\textrm{for all}\;\;x\in\textrm{Conv}\left(\left\{\;q\in\mathbb{R}^d\;|\;\overline{H}(p+q)\leq\overline{H}(p)\;\right\}\right).\end{equation}

Let $\overline{m}\in\Lip(\mathbb{R}^d)$ denote the function constructed in Proposition \ref{metric_limit} corresponding to $p\in\mathbb{R}^d$, and, by Proposition \ref{metric_homogenization}, satisfying $$\overline{H}(p+D\overline{m})\leq \overline{H}(p)\;\;\textrm{on}\;\;\mathbb{R}^d.$$  We write $\rho_\epsilon:=\epsilon^{-d}\rho(x/\epsilon)$ to denote a rescaling of a standard, compactly supported, nonnegative, radially symmetric, smooth convolution kernel $\rho:\mathbb{R}^d\rightarrow[0,\infty)$ and define, for each $\epsilon>0$, $\overline{m}^\epsilon=\overline{m}\ast\rho_\epsilon.$  Proposition \ref{appendix_regular} of the appendix and (\ref{metric_ray_1}) imply that, for each $\epsilon>0$ and $t\in\mathbb{R}$, $$\frac{d}{dt}\overline{m}^\epsilon(t\nu)=D\overline{m}^\epsilon(t\nu)\cdot\nu\leq 0.$$  Therefore, using Proposition \ref{metric_limit}, for each $\epsilon>0$ and $t\geq 0$, $$0\leq \overline{m}^\epsilon(t\nu)\leq \overline{m}^\epsilon(0).$$  After passing to the limit, as $\epsilon\rightarrow 0$, we conclude that, for each $t\geq 0$, $$\overline{m}(t\nu)=\overline{m}(0)=0,$$ thereby completing the proof.  \end{proof}

\section{The Almost Sure Characterization of the Effective Hamiltonian}

We now complete our characterization of the effective Hamiltonian for gradients $p\in\mathbb{R}^d$ satisfying either,\begin{equation}\label{sure_assumption} \overline{H}(p)=\min_{q\in\mathbb{R}^d}\overline{H}(q),\;\;\textrm{or,}\;\; p\in\partial\left(\left\{\;\overline{H}(q)\leq\overline{H}(p)\;\right\}\right)\cap\partial\left(\textrm{Conv}\left(\left\{\;\overline{H}(q)\leq\overline{H}(p)\;\right\}\right)\right),\end{equation} by proving that there exists a subset of full probability such that, for every $p\in\mathbb{R}^d$ satisfying $(\ref{sure_assumption})$, for $v^\delta$ the solution of (\ref{macroscopic}) corresponding to $p\in\mathbb{R}^d$, for each $R>0$, $$\lim_{\delta\rightarrow 0}\sup_{x\in B_{R/\delta}}\abs{\overline{H}(p)+\delta v^\delta(x,\omega)}=0.$$  The following proposition proves that, in analogy to the $\limsup$ of the $-\delta v^\delta$, the $\liminf$ of the $-\delta v^\delta$ is deterministic and characterized identically.  The proof is as in Proposition \ref{effective_ball}.

\begin{prop}\label{sure_ball} Assume (\ref{steady}).  There exists subset $\Omega_4\subset\Omega_2$ of full probability such that, for each $\omega\in\Omega_4$, $p\in\mathbb{R}^d$, $y\in\mathbb{R}^d$ and $R>0$, for $v^\delta$ the solution of (\ref{macroscopic}) corresponding to $p\in\mathbb{R}^d$, $$\overline{H}(p)=\limsup_{\delta\rightarrow 0}-\delta v^\delta(y/\delta,\omega)=\limsup_{\delta\rightarrow 0}\sup_{x\in B_{R/\delta}(y/\delta)}-\delta v^\delta(x,\omega),$$ and, $$\tilde{H}(p):=\liminf_{\delta\rightarrow 0}-\delta v^\delta(y/\delta,\omega)=\liminf_{\delta\rightarrow 0}\inf_{x\in B_{R/\delta}(y/\delta)}-\delta v^\delta(x,\omega).$$\end{prop}

Our goal, therefore, is to prove that, for each $p\in\mathbb{R}^d$ satisfying (\ref{sure_assumption}), we have $\overline{H}(p)=\tilde{H}(p).$  We begin with elements of the minimal level set of the effective Hamiltonian.

\begin{prop}\label{sure_minimum}  Assume (\ref{steady}).  For each $p\in\mathbb{R}^d$ satisfying \begin{equation}\label{sure_minimum_1}\overline{H}(p)=\min_{q\in\mathbb{R}^d}\overline{H}(q),\end{equation}  for each $\omega\in\Omega_4$, for each $R>0$, for $v^\delta$ the solution of (\ref{macroscopic}) corresponding to $p\in\mathbb{R}^d$, $$\lim_{\delta\rightarrow 0}\sup_{x\in B_{R/\delta}}\abs{\overline{H}(p)+\delta v^\delta(x,\omega)}=0.$$\end{prop}

\begin{proof}  Fix $p\in\mathbb{R}^d$ satisfying (\ref{sure_minimum_1}).  We proceed by contradiction.  Suppose that \begin{equation}\label{sure_minimum_2} \tilde{H}(p)<\overline{H}(p)=\min_{q\in\mathbb{R}^d}\overline{H}(q).\end{equation}  For each $\delta>0$ and $\omega\in\Omega$, we write, for $v^\delta$ the solution of (\ref{macroscopic}) corresponding to $p$, $$w^\delta(x,\omega)=v^\delta(x,\omega)-v^\delta(0,\omega).$$  In view of Proposition \ref{macroscopic_posed}, for ever $\omega\in\Omega$ there exists a subsequence $\left\{\delta_k=\delta_k(\omega)\rightarrow 0\right\}_{k=1}^\infty$ and $w\in\Lip(\mathbb{R}^d)$ such that, as $k\rightarrow\infty$, \begin{equation}\label{sure_minimum_3} -\delta_k v^{\delta_k}(0,\omega)\rightarrow \liminf_{\delta\rightarrow 0}-\delta v^\delta(0,\omega)\;\;\textrm{and}\;\;w^{\delta_k}\rightarrow w\;\;\textrm{locally uniformly on}\;\;\mathbb{R}^d.\end{equation}  The stability of viscosity solutions, see \cite{CIL}, and Proposition \ref{macroscopic_posed} imply that \begin{equation}\label{sure_minimum_4} H(p+Dw,x,\omega)=\liminf_{\delta\rightarrow 0}-\delta v^\delta(0,\omega)\;\;\textrm{on}\;\;\mathbb{R}^d.\end{equation}  In particular, for every $\omega\in\Omega_4$, \begin{equation}\label{sure_minimum_6} H(p+Dw,x,\omega)=\tilde{H}(p)\;\;\textrm{on}\;\;\mathbb{R}^d.\end{equation}

We now proceed as in Proposition \ref{metric_limit} and Proposition \ref{metric_homogenization}.  Using (\ref{coercive}), fix $C_1>0$ larger than the constant appearing in Proposition \ref{macroscopic_posed} and such that, for every $\omega\in\Omega$, the map \begin{equation}\label{sure_minimum_5} x\rightarrow C_1\abs{x}\;\;\textrm{is a supersolution of (\ref{sure_minimum_4}) on}\;\;\mathbb{R}^d\setminus\left\{0\right\}.\end{equation} 

Define, for each $\omega\in\Omega$ and $x,y\in\mathbb{R}^d$, \begin{multline}\label{sure_minimum_11}\mathcal{S}(x,y,\omega)=\left\{\;z(x)-z(y)\;| z(\cdot)\in\C(\mathbb{R}^d),\;\;z(x)-z(y)\leq C_1\abs{x-y}\right. \\ \left.\textrm{and}\;\;z\;\;\textrm{is a subsolution of}\;(\ref{sure_minimum_4})\;\textrm{for this}\; \omega\in\Omega.\;\right\}.\end{multline}  Notice that, for each $x,y\in\mathbb{R}^d$ and $\omega\in\Omega$, $\mathcal{S}(x,y,\omega)\neq \emptyset$ by (\ref{sure_minimum_3}) and (\ref{sure_minimum_4}).  Observe, again, the remark following the definition of $\mathcal{A}(x,y,\omega)$ in Proposition \ref{metric_posed}.

For each $x,y\in\mathbb{R}^d$ and $\omega\in\Omega$, define $$n(x,y,\omega)=\sup\mathcal{S}(x,y,\omega).$$  By repeating the arguments appearing in Proposition \ref{metric_posed}, we conclude that $n:\mathbb{R}^d\times\mathbb{R}^d\times\Omega\rightarrow\mathbb{R}$ satisfies (\ref{metric_eq}) with, for each $\omega\in\Omega_4$, righthand side $\tilde{H}(p)$, (\ref{metric_sub}), (\ref{metric_joint}) and (\ref{metric_bound}).  Therefore, by repeating the arguments in Proposition \ref{metric_limit} and Proposition \ref{metric_homogenization}, we conclude that there exists $\overline{n}\in\Lip(\mathbb{R}^d)$ satisfying $$\overline{H}(p+D\overline{n})\leq \tilde{H}(p)\;\;\textrm{on}\;\;\mathbb{R}^d,$$ contradicting (\ref{sure_minimum_2}).  Since, for every $p\in\mathbb{R}^d$ satisfying (\ref{sure_minimum_1}), this implies that $\tilde{H}(p)=\overline{H}(p)$, the result follows from Proposition \ref{sure_ball}.  \end{proof}

We now consider the case that $p\in\mathbb{R}^d$ satisfies \begin{equation}\label{sure_assumption_1} p\in\partial\left(\left\{\;\overline{H}(q)\leq\overline{H}(p)\;\right\}\right)\cap\partial\left(\textrm{Conv}\left(\left\{\;\overline{H}(q)\leq\overline{H}(p)\;\right\}\right)\right).\end{equation}  The intuition in this case is as follows.  Suppose that, for $p\in\mathbb{R}^d$ satisfying (\ref{sure_assumption_1}) and $\omega\in\Omega$, there exists a subsequence $\left\{\delta_k\rightarrow 0\right\}_{k=1}^\infty$ satisfying, as $k\rightarrow\infty$, \begin{equation}\label{sure_int_1}-\delta_k v^{\delta_k}(0,\omega)\rightarrow\tilde{H}(p)<\overline{H}(p).\end{equation}  Then, proceeding in a purely formal manner, there exists $\overline{v}\in\Lip(\mathbb{R}^d)$ satisfying, as $k\rightarrow\infty$, \begin{equation}\label{sure_int_2}-\delta v^{\delta_k}(x/\delta_k,\omega)\rightarrow\overline{v}\;\;\textrm{locally uniformly on}\;\;\mathbb{R}^d\;\;\textrm{with}\;\;\overline{H}(p+D\overline{v})=\tilde{H}(p)\;\;\textrm{on}\;\;\mathbb{R}^d.\end{equation}

We obtain a contradiction as follows.  For $\nu$ as in Proposition \ref{metric_ray}, (\ref{sure_int_2}) implies that $\overline{v}$ is strictly decreasing on every ray parallel to $\nu$ in the direction of $\nu$.  However, this is impossible since (\ref{sure_int_1}) and Proposition \ref{sure_ball} imply that $\overline{v}$ achieves a global maximum at zero.

The argument itself follows via a comparison argument with $\overline{m}$ constructed in Proposition \ref{metric_limit} corresponding to $p\in\mathbb{R}^d$, using the fact that $\overline{m}$ vanishes along the ray emanating from the origin in direction $\nu$.  In the case that $p\in\mathbb{R}^d$ fails to satisfy (\ref{sure_assumption_1}), it is possible for $\overline{m}$ to be strictly increasing along every ray emanating from the origin, thereby masking the behavior of the solutions $v^\delta$ to (\ref{macroscopic}), and compromising the comparison argument.  Finally, we remark that this argument is motivated by the analogous fact seen in \cite{AS}.

\begin{prop}\label{sure_boundary}  Assume (\ref{steady}).  For every $p\in\mathbb{R}^d$ satisfying \begin{equation}\label{sure_boundary_1} p\in\partial\left(\left\{\;\overline{H}(q)\leq\overline{H}(p)\;\right\}\right)\cap\partial\left(\textrm{Conv}\left(\left\{\;\overline{H}(q)\leq\overline{H}(p)\;\right\}\right)\right),\end{equation} there exists $\Omega_5=\Omega_5(p)\subset\Omega$ of full probability such that, for every $\omega\in\Omega_5$ and $R>0$, for $v^\delta$ the solution of (\ref{macroscopic}) corresponding to $p\in\mathbb{R}^d$, $$\lim_{\delta\rightarrow 0}\sup_{x\in B_{R/\delta}}\abs{\overline{H}(p)+\delta v^\delta(x,\omega)}=0.$$\end{prop}

\begin{proof}  Fix $p\in\mathbb{R}^d$ satisfying (\ref{sure_boundary_1}).  Fix $\omega\in \Omega_3(p)\cap\Omega_4$ and write $v^\delta$ for the solution of (\ref{macroscopic}) corresponding to $p$.  We proceed by contradiction.  Suppose that $\tilde{H}(p)<\overline{H}(p)$ and choose a subsequence $\left\{\delta_k=\delta_k(\omega)\rightarrow 0\right\}_{k=1}^\infty$ satisfying, as $k\rightarrow\infty$, \begin{equation}\label{sure_proof_14} -\delta_k v^{\delta_k}(0,\omega)\rightarrow\tilde{H}(p).\end{equation}  Define, for each $1\leq k<\infty$, for $\eta>0$ to be fixed later and $\nu\in\mathbb{R}^d$ as in Proposition \ref{metric_ray} corresponding to $p$, $$\tilde{v}^{\delta_k}(x)=\delta_k v^{\delta_k}(x/\delta_k)-\eta\abs{x}^2\;\;\textrm{and}\;\;m^{\delta_k}(x)=\delta_k m(x/\delta_k,-\nu/\delta_k,\omega),$$ and observe that since $\tilde{H}(p)<\overline{H}(p)$, in view of (\ref{hamcon}), Proposition \ref{macroscopic_posed} and (\ref{sure_proof_14}), for $\eta>0$ and $0<R<1$ sufficiently small, there exists $\overline{k}\geq 1$ and $C>0$, such that, for all $k\geq\overline{k}$, $$H(p+D\tilde{v}^{\delta_k},w/\delta_k,\omega)\leq \overline{H}(p)-C\;\;\textrm{on}\;\;B_R.$$  Since, for each $1\leq k<\infty$, Proposition \ref{metric_posed} implies that $m^{\delta_k}$ satisfies $$H(p+Dm^{\delta_k},x/\delta_k,\omega)=\overline{H}(p)\;\;\textrm{on}\;\;\mathbb{R}^d\setminus\left\{-\nu\right\},$$ the comparison principle, see \cite{CIL}, and $\abs{\nu}=1$ with $0<R<1$ imply, for all $k\geq \overline{k}$, \begin{equation}\label{sure_proof_15} \sup_{x\in B_R}(\tilde{v}^{\delta_k}(x)-m^{\delta_k}(x))=\sup_{x\in\partial B_R}(\tilde{v}^{\delta_k}(x)-m^{\delta_k}(x)).\end{equation}

We prove now that (\ref{sure_proof_15}) yields a contradiction.  Observe that, since $\omega\in\Omega_3(p)$, as $k\rightarrow\infty$, $$m^{\delta_k}(x)\rightarrow\overline{m}(x+\nu)\;\;\textrm{locally uniformly on}\;\;\mathbb{R}^d,$$ with $m(x+\nu)\geq 0$ on $\mathbb{R}^d$ by Proposition \ref{metric_limit} and $\overline{m}(0+\nu)=\overline{m}(\nu)=0$ by Proposition \ref{metric_ray}.  And, since $\omega\in\Omega_4$, Proposition \ref{sure_ball} and (\ref{sure_proof_14}) imply $$\limsup_{k\rightarrow\infty}\sup_{x\in\partial B_R}\tilde{v}^{\delta_k}(x)\leq -\tilde{H}(p)-\eta R^2\;\;\textrm{with}\;\;\lim_{k\rightarrow \infty}\tilde{v}^{\delta_k}(0)=-\tilde{H}(p).$$  Therefore, $$\lim_{k\rightarrow\infty}\sup_{x\in B_R}(\tilde{v}^{\delta_k}(x)-m^{\delta_k}(x))= -\tilde{H}(p)>-\tilde{H}(p)-\eta R^2\geq\limsup_{k\rightarrow\infty}\sup_{x\in\partial B_R}(\tilde{v}^{\delta_k}(x)-m^{\delta_k}(x)),$$ contradicting (\ref{sure_proof_15}).  We therefore define the subset of full probability $$\Omega_5(p)=\Omega_3(p)\cap\Omega_4,$$ and conclude that, for each $\omega\in\Omega_5$, as $\delta\rightarrow 0$, $$-\delta v^\delta(0,\omega)\rightarrow \overline{H}(p).$$  Since $p\in\mathbb{R}^d$ satisfying (\ref{sure_boundary_1}) was arbitrary, this implies that for every such $p\in\mathbb{R}^d$ we have $\tilde{H}(p)=\overline{H}(p)$.  The result then follows by Proposition \ref{sure_ball}.  \end{proof}

We conclude the section with the proof of our main result, which is essentially a corollary of Proposition \ref{sure_minimum} and Proposition \ref{sure_boundary}.

\begin{thm}\label{sure_main} Assume (\ref{steady}).  There exists a subset $\Omega_6\subset\Omega$ of full probability such that, for each $\omega\in\Omega_6$, for each $p\in\mathbb{R}^d$ satisfying \begin{equation}\label{sure_main_1} \overline{H}(p)=\min_{q\in\mathbb{R}^d}\overline{H}(q),\;\;\textrm{or,}\;\; p\in\partial\left(\left\{\;\overline{H}(q)\leq\overline{H}(p)\;\right\}\right)\cap\partial\left(\textrm{Conv}\left(\left\{\;\overline{H}(q)\leq\overline{H}(p)\;\right\}\right)\right),\end{equation} for each $R>0$, for $v^\delta$ the solution of (\ref{macroscopic}) corresponding to $p\in\mathbb{R}^d$, $$\lim_{\delta\rightarrow 0}\sup_{x\in B_{R/\delta}}\abs{\overline{H}(p)+\delta v^\delta(x,\omega)}=0.$$\end{thm}

\begin{proof}  We define $$\mathcal{H}=\left\{\;p\in\mathbb{R}^d\;|\;p\;\textrm{satisfies}\;(\ref{sure_main_1})\;\right\}.$$  For each $p\in\mathcal{H}$, if $$\overline{H}(p)=\min_{q\in\mathbb{R}^d}\overline{H}(q)\;\;\textrm{define}\;\;\Omega_5(p):=\Omega_4.$$  Else, define $\Omega_5(p)$ as in Proposition \ref{sure_boundary}.

Choose a countable, dense subset $\left\{p_i\right\}_{i=1}^\infty\subset\mathcal{H}$ and define the subset of full probability $$\Omega_6=\bigcap_{i=1}^\infty\Omega_5(p_i).$$  The claim follows by density, Proposition \ref{macroscopic_posed}, Proposition \ref{macroscopic_continuity}, Proposition \ref{sure_minimum} and Proposition \ref{sure_boundary}.  \end{proof}

\section{A Proof of Homogenization}

We now present a proof of homogenization under the assumption that the set $$\mathcal{H}:=\left\{\;p\in\mathbb{R}^d\;|\;p\;\textrm{satisfies}\;(\ref{sure_main_1})\;\right\}$$ encompasses the whole space.  \begin{equation}\label{assume_space}  \textrm{Assume}\;\mathcal{H}=\mathbb{R}^d.\end{equation}  We remark that condition (\ref{assume_space}) is satisfied whenever, for each $y\in\mathbb{R}^d$ and $\omega\in\Omega$, the map $$p\rightarrow H(p,y,\omega)\;\;\textrm{is convex.}$$  In this case, a comparison argument similar to those presented in Proposition \ref{effective_properties} proves that $\overline{H}:\mathbb{R}^d\rightarrow\mathbb{R}$ is convex and, therefore, satisfies (\ref{assume_space}).  See, for instance, \cite{AS} or \cite{F}.  Also, in \cite{AS2}, homogenization is obtained for a family of first-order, level-set convex Hamilton-Jacobi equations.  We remark that their assumptions ensure that the effective Hamiltonian satisfies (\ref{assume_space}).

We recall, for each $\epsilon>0$ and $\omega\in\Omega$, the equation \begin{equation}\label{homogenization_eq} \left\{\begin{array}{ll} u^\epsilon_t+H(Du^\epsilon,x/\epsilon,\omega)=0 & \textrm{on}\;\;\mathbb{R}^d\times(0,\infty), \\ u^\epsilon=u_0 & \textrm{on}\;\;\mathbb{R}^d\times\left\{0\right\},\end{array}\right.\end{equation} and the effective equation \begin{equation}\label{homogenization_effective} \left\{\begin{array}{ll} \overline{u}_t+\overline{H}(D\overline{u})=0 & \textrm{on}\;\;\mathbb{R}^d\times(0,\infty), \\ \overline{u}=u_0 & \textrm{on}\;\;\mathbb{R}^d\times\left\{0\right\}.\end{array}\right.\end{equation}  We will prove that under, assumption (\ref{assume_space}), for each $\omega\in\Omega_6$, as $\epsilon\rightarrow 0$, $$u^\epsilon\rightarrow \overline{u}\;\;\textrm{locally uniformly on}\;\;\mathbb{R}^d.$$

The following two propositions are basic facts from the theory of viscosity solutions, see \cite{CIL}.  The first states that (\ref{homogenization_eq}) is well-posed, and that the solutions can inherit some regularity from the initial condition.  The second proposition states the standard contraction properties of (\ref{homogenization_eq}) and (\ref{homogenization_effective}).  In each case, recall (\ref{steady}) and Proposition \ref{effective_properties}.

\begin{prop}\label{homogenization_posed}  Assume (\ref{steady}).  For each $\epsilon>0$ and $\omega\in\Omega$, there exists a unique solution $u^\epsilon:\mathbb{R}^d\times[0,\infty)\rightarrow\mathbb{R}$ of (\ref{homogenization_eq}) satisfying $$u^\epsilon\in\BUC(\mathbb{R}^d\times[0,T))$$ for each $T>0$.  Furthermore, if $u_0\in\Lip(\mathbb{R}^d)$, there exists $C=C(\norm{Du_0}_{L^\infty(\mathbb{R}^d)})$ such that, for each $\epsilon>0$ and $\omega\in\Omega$, $$\norm{u^\epsilon_t}_{L^\infty(\mathbb{R}^d\times[0,\infty))}\leq C\;\;\textrm{and}\;\;\norm{Du^\epsilon}_{L^\infty(\mathbb{R}^d\times[0,\infty);\mathbb{R}^d)}\leq C.$$\end{prop}

\begin{prop}\label{homogenization_contraction}  Assume (\ref{steady}).  For every $u_0,v_0\in\BUC(\mathbb{R}^d)$, $\omega\in\Omega$ and $\epsilon>0$, let $u^\epsilon$ and $v^\epsilon$ denote the solutions of (\ref{homogenization_eq}) with initial conditions $u_0$ and $v_0$ respectively, and let $\overline{u}$ and $\overline{v}$ denote the solutions of (\ref{homogenization_effective}) with initial conditions $u_0$ and $v_0$ respectively.  Then, $$\norm{u^\epsilon-v^\epsilon}_{L^\infty(\mathbb{R}^d\times[0,\infty))}\leq \norm{u_0-v_0}_{L^\infty(\mathbb{R}^d)}\;\;\textrm{and}\;\;\norm{\overline{u}-\overline{v}}_{L^\infty(\mathbb{R}^d\times[0,\infty))}\leq \norm{u_0-v_0}_{L^\infty(\mathbb{R}^d)}.$$\end{prop}

We are now prepared to present the proof of homogenization.  The result will follow from Theorem \ref{sure_main} and the standard perturbed test function method.

\begin{thm}\label{homogenization_proof}  Assume (\ref{steady}) and (\ref{assume_space}).  For each $\omega\in\Omega_6$, for $u^\epsilon$ the solutions of (\ref{homogenization_eq}) corresponding to $\omega\in\Omega$ and $\overline{u}$ the solution of (\ref{homogenization_effective}), as $\epsilon\rightarrow 0$, $$u^\epsilon\rightarrow\overline{u}\;\;\textrm{locally uniformly on}\;\;\mathbb{R}^d\times[0,\infty).$$ \end{thm}

\begin{proof}  We first consider the case that \begin{equation}\label{homogenization_proof_1} u_0\in \Lip(\mathbb{R}^d).\end{equation}  In this case, Proposition \ref{homogenization_posed} implies that, after passing to a subsequence $\left\{\epsilon_k\rightarrow 0\right\}_{k=1}^\infty$, for $u\in\Lip(\mathbb{R}^d\times[0,\infty))$, as $k\rightarrow\infty$, \begin{equation}\label{homogenization_proof_2}u^{\epsilon_k}\rightarrow u\;\;\textrm{locally uniformly on}\;\;\mathbb{R}^d\times[0,\infty).\end{equation}  We will show that $u=\overline{u}$ satisfies (\ref{homogenization_effective}) with initial condition $u_0$.

First, notice that the convergence (\ref{homogenization_proof_2}) implies that \begin{equation}\label{homogenization_proof_4} u=u_0\;\;\textrm{on}\;\;\mathbb{R}^d\times\left\{0\right\}.\end{equation}  Now, proceeding by contradiction, suppose for $\phi\in\C^2(\mathbb{R}^d\times(0,\infty))$ and $(x_0,t_0)\in\mathbb{R}^d\times(0,\infty)$, \begin{equation}\label{homogenization_proof_5} u-\phi\;\;\textrm{has a strict local maximum at}\;(x_0,t_0)\;\textrm{with}\;\phi_t(x_0,t_0)+\overline{H}(D\phi(x_0,t_0))=\theta>0.\end{equation}  Furthermore, fix $\overline{r}>0$ such that, for each $0<r<\overline{r}$, we have \begin{equation}\label{homogenization_proof_7} B_r((x_0,t_0))\subset\mathbb{R}^d\times(0,\infty)\;\;\textrm{with}\;\;\sup_{(x,t)\in \partial B_r(x_0,t_0)}(u(x,t)-\phi(x,t))<\sup_{(x,t)\in B_r(x_0,t_0)}(u(x,t)-\phi(x,t)).\end{equation}

For each $\delta>0$, let $v^\delta$ denote the solution of (\ref{macroscopic}) corresponding to $D\phi(x_0,t_0)$.  Define, for each $1\leq k<\infty$, the perturbed test function \begin{equation}\label{homogenization_proof_6}\phi_k(x)=\phi(x)+\epsilon_kw^{\epsilon_k}(x/\epsilon_k,\omega)\;\;\textrm{for}\;\;w^{\epsilon_k}(x)=v^{\epsilon_k}(x,\omega)-v^{\epsilon_k}(0,\omega).\end{equation}  We will prove that there exists $0<r\leq\overline{r}$ and $\overline{k}\geq 1$ such that, for each $k\geq \overline{k}$, (\ref{homogenization_proof_5}) implies that $\phi_k$ is a supersolution of (\ref{homogenization_eq}) on $B_r((x_0,t_0)).$

Suppose that, for $\eta\in\C^2(\mathbb{R}^d\times(0,\infty))$ and $(y_0,s_0)\in\mathbb{R}^d\times(0,\infty)$, $$\phi_k-\eta\;\;\textrm{has a local minimum at}\;(y_0,s_0).$$  Then, $$w^{\epsilon_k}(x)-\frac{1}{\epsilon_k}(\eta(\epsilon_k x,\epsilon _kt)-\phi(\epsilon_k x,\epsilon_k t))\;\;\textrm{has a local minimum at}\;(y_0/\epsilon_k,s_0/\epsilon_k).$$  Therefore, in view of (\ref{homogenization_proof_6}), after returning to the original scaling, $$H(D\phi(x_0,t_0)-D\phi(y_0,s_0)+D\eta(y_0,s_0),y_0/\epsilon_k,\omega)\geq -\epsilon_k v^{\epsilon_k}(y_0/\epsilon_k,\omega).$$  And, since $\eta_t(y_0,s_0)=\phi_t(y_0,s_0)$, $$\eta_t(y_0,s_0)+H(D\phi(x_0,t_0)-D\phi(y_0,s_0)+D\eta(y_0,s_0),y_0/\epsilon_k,\omega)\geq \phi_t(y_0,s_0)-\epsilon_k v^{\epsilon_k}(y_0/\epsilon_k,\omega).$$  Therefore, in view of (\ref{hamcon}), Proposition \ref{macroscopic_posed}, Theorem \ref{sure_main}, $\omega\in\Omega_6$ and $\phi\in \C^2(\mathbb{R}^d)$, our assumption (\ref{homogenization_proof_5}) implies that there exists $\overline{k}\geq 1$ and $0<r<\overline{r}$ sufficiently small such that, whenever $(y_0,s_0)\in B_r((x_0,t_0))$ and $k\geq \overline{k}$, $$\eta_t(y_0,s_0)+H(D\eta(y_0,s_0),y_0/\epsilon_k,\omega)\geq\frac{\theta}{2}>0.$$

We therefore conclude that, for all $k\geq \overline{k}$, $\phi_k$ is a strict supersolution of (\ref{homogenization_eq}) on $B_r(x_0,t_0)$.  The comparison principle, see \cite{CIL}, implies that, for each $k\geq \overline{k}$, $$\sup_{(x,t)\in B_r(x_0,t_0)}(u^{\epsilon_k}(x,t)-\phi_k(x,t))=\sup_{(x,t)\in \partial B_r(x_0,t_0)}(u^{\epsilon_k}(x,t)-\phi_k(x,t)).$$  Since Theorem \ref{sure_main}, $\omega\in\Omega_6$ and (\ref{homogenization_proof_2}) imply that, as $k\rightarrow\infty$, $$\phi_k\rightarrow\phi\;\;\textrm{and}\;\;u^{\epsilon_k}\rightarrow u\;\;\textrm{locally uniformly on}\;\;\mathbb{R}^d,$$ we conclude that $$\sup_{(x,t)\in B_r(x_0,t_0)}(u(x,t)-\phi(x,t))=\sup_{(x,t)\in\partial B_r(x_0,t_0)}(u(x,t)-\phi(x,t)),$$  a contradiction of (\ref{homogenization_proof_7}) since $0<r<\overline{r}$.  The analogous argument proves that $u$ is a supersolution of (\ref{homogenization_effective}) and, therefore, that $u=\overline{u}$ is the unique solution of (\ref{homogenization_effective}).  And, since the subsequence $\left\{\epsilon_k\rightarrow0\right\}_{k=1}^\infty$ was arbitrary, this completes the proof under assumption (\ref{homogenization_proof_1}).

We now consider $u_0\in\BUC(\mathbb{R}^d)$.  Fix $\omega\in\Omega_6$.  Let $\rho_\eta:=\eta^{-d}\rho(x/\eta)$ to denote a scaling of a standard, nonnegative, radially symmetric, smooth convolution kernel $\rho:\mathbb{R}^d\rightarrow[0,\infty)$, and define, for each $\eta>0$, the convolution $$u_0^\eta=u_0\ast\rho_\eta.$$  For each $\epsilon>0$ and $\eta>0$, let $u^\epsilon$ and $u^{\epsilon,\eta}$ denote the solutions of (\ref{homogenization_eq}) with initial conditions $u_0$ and $u_0^\eta$ respectively and, let $\overline{u}$ and $\overline{u}^\eta$ denote the solutions of (\ref{homogenization_effective}) with initial conditions $u_0$ and $u_0^\eta$ respectively.  Observe that, since $u_0\in\BUC(\mathbb{R}^d)$, for each $\eta>0$ we have $u_0^\eta\in\Lip(\mathbb{R}^d)$ and, as $\eta\rightarrow 0$, \begin{equation}\label{homogenization_proof_9} u_0^\eta\rightarrow u_0\;\;\textrm{uniformly on}\;\;\mathbb{R}^d.\end{equation}

Let $K\subset\mathbb{R}^d\times[0,\infty)$ be a compact subset.  Then, for each $\epsilon>0$, $$\norm{u^\epsilon-\overline{u}}_{L^\infty(K)}\leq \norm{u^\epsilon-u^{\epsilon,\eta}}_{L^\infty(K)}+\norm{u^{\epsilon,\eta}-\overline{u}^\eta}_{L^\infty(K)}+\norm{\overline{u}^\eta-\overline{u}}_{L^\infty(K)}.$$  Therefore, in view of Proposition \ref{homogenization_contraction}, since $u_0^\eta\in\Lip(\mathbb{R}^d)$, $$\limsup_{\epsilon\rightarrow 0}\norm{u^\epsilon-\overline{u}}_{L^\infty(K)}\leq 2\norm{u_0^\eta-u_0}_{L^\infty(\mathbb{R}^d)}.$$ Since $\eta>0$ was arbitrary,  (\ref{homogenization_proof_9}) implies that, for each $\omega\in\Omega_6$, $$\lim_{\epsilon\rightarrow 0}\norm{u^\epsilon-\overline{u}}_{L^\infty(K)}=0.$$  Since $K\subset\mathbb{R}^d\times[0,\infty)$ was arbitrary, this completes the argument.  \end{proof}

\section{A Remark on Viscous Hamilton-Jacobi Equations}

In this section, we describe the modifications of the above proof which are necessary in order to treat equations of the form \begin{equation}\label{viscous_eq} \left\{\begin{array}{ll} u^\epsilon_t-\epsilon\tr(A(x/\epsilon,\omega)D^2u^\epsilon)+H(Du^\epsilon,u^\epsilon,x,x/\epsilon,\omega)=0 & \textrm{on}\;\;\mathbb{R}^d\times(0,\infty), \\ u^\epsilon=u_0 & \textrm{on}\;\;\mathbb{R}^d\times\left\{0\right\}.\end{array}\right.\end{equation}  In this setting, the approximate macroscopic problem becomes, for each $p\in\mathbb{R}^d$, $r\in\mathbb{R}$, $x\in\mathbb{R}^d$ and $\omega\in\Omega$, \begin{equation}\label{viscous_macroscopic} \delta v^\delta -\tr(A(y,\omega)D^2 v^\delta)+H(p+Dv^\delta, r,x,y,\omega)=0\;\;\textrm{on}\;\;\mathbb{R}^d,\end{equation} the effective Hamiltonian $\overline{H}:\mathbb{R}^d\times\mathbb{R}\times\mathbb{R}^d\rightarrow\mathbb{R}$ is identified on a subset of full probability by, for $v^\delta$ the solution of (\ref{viscous_macroscopic}) corresponding to $p\in\mathbb{R}^d$, $r\in\mathbb{R}$ and $x\in\mathbb{R}^d$, \begin{equation}\label{viscous_effective} \overline{H}(p,r,x)=\limsup_{\delta\rightarrow 0}-\delta v^\delta(0,\omega),\end{equation} and the corresponding metric problem is of the type, for each $p\in\mathbb{R}^d$, $r\in\mathbb{R}$, $x\in\mathbb{R}^d$ and $\omega\in\Omega$, \begin{equation}\label{viscous_metric}-\tr(A(y,\omega)D^2m)+H(p+Dm,r,x,y,\omega)=\overline{H}(p,r,x)\;\;\textrm{on}\;\;\mathbb{R}^d\setminus\left\{0\right\}.\end{equation}

See, for instance, \cite{LS1}, \cite{AS} (under additional coercivity and boundedness assumptions like (\ref{coercive}) and (\ref{bounded})) or \cite{F} for a few examples of sufficient assumptions required to treat equations like (\ref{viscous_eq}).  The only changes required in the above proof occur in Proposition \ref{macroscopic_posed}, Proposition \ref{metric_posed} and Proposition \ref{sure_minimum}.  The gradient estimates occurring in Proposition \ref{macroscopic_posed} and Proposition \ref{metric_posed} follow from a Bernstein argument, see for instance \cite{LS1} or \cite{AS}.  And, in Proposition \ref{metric_posed} and Proposition \ref{sure_minimum}, we need to modify the supersolution occurring (\ref{metric_posed_4}) and (\ref{sure_minimum_11}) to account for the second-order term.  See first \cite{AS}, and again \cite{F} for the details.  Furthermore, see \cite{F} for the necessary assumptions required to treat the time-independent version of (\ref{viscous_eq}).

We then have the following theorem, and the analogous time-independent analogue which can be stated identically.

\begin{thm}\label{viscous_main}  Assume, for instance, what appears \cite{LS1}, \cite{AS} (under additional coercivity and boundedness assumptions like (\ref{coercive}) and (\ref{bounded})) or \cite{F}.  There exists a subset of full probability such that, for every $r\in\mathbb{R}$ and $x\in\mathbb{R}^d$, whenever $p\in\mathbb{R}^d$ satisfies \begin{multline*} \overline{H}(p,r,x)=\min_{q\in\mathbb{R}^d}\overline{H}(q,r,x),\;\;\textrm{or,} \\ p\in\partial\left(\left\{\;\overline{H}(q,r,x)\leq\overline{H}(p,r,x)\;\right\}\right)\cap\partial\left(\textrm{Conv}\left(\left\{\;\overline{H}(q,r,x)\leq\overline{H}(p,r,x)\;\right\}\right)\right),\end{multline*} we have, for each $R>0$, for $v^\delta$ the solution of (\ref{viscous_macroscopic}) corresponding to $p\in\mathbb{R}^d$, $r\in\mathbb{R}$ and $x\in\mathbb{R}^d$, $$\lim_{\delta\rightarrow 0}\sup_{y\in B_{R/\delta}}\abs{\overline{H}(p,r,x)+\delta v^\delta(y,\omega)}=0.$$\end{thm}

\section{An Application to First Order Equations}

In this section we provide an application of Theorem \ref{homogenization_proof} applying to a general class of first-order Hamilton-Jacobi equations.  Essentially, we prove that if the nonconvexity of the Hamiltonian $H(p,y,\omega)$ is localized in the gradient variable $p\in\mathbb{R}^d$, then the failure of homogenization is localized to a bounded, open subset of $\mathbb{R}^d$.  Furthermore, we prove that, insofar as homogenization is concerned, this situation is generic.  That is, if homogenization is true for such Hamiltonians, then homogenization is true in complete generality.

The proof follows from the fact that, in the first-order setting, the solutions to the macroscopic problem have a localized dependence on the gradient variable of $H(p,y,\omega)$.  The failure of this argument in the second-order case is representative of the move from an optimal control problem to a stochastic control problem.  Compare, for instance, \cite{S} and \cite{LS1}.

Assume that there exists $\tilde{H}(p,y,\omega):\mathbb{R}^d\times\mathbb{R}^d\times\Omega\rightarrow\mathbb{R}$ satisfying the relevant assumptions of (\ref{steady}) such that, for some $R>0$, \begin{equation}\label{application_equal} \tilde{H}(p,y,\omega)=H(p,y,\omega)\;\;\textrm{on}\;\; \left(\mathbb{R}^d\setminus B_R\right)\times\mathbb{R}^d\times\Omega,\end{equation} and such that, for each $y\in\mathbb{R}^d$ and $\omega\in\Omega$, \begin{equation}\label{application_convex} p\rightarrow \tilde{H}(p,y,\omega)\;\;\textrm{is convex.}\end{equation}  This condition is satisfied, for instance, by the Hamiltonian considered in \cite{ATY}.

\begin{prop}\label{application_bounded}  Assume (\ref{steady}), (\ref{application_equal}) and (\ref{application_convex}).  There exists a subset $\Omega_7\subset\Omega$ of full probability and a bounded, open subset $U\subset\mathbb{R}^d$ such that, for every $\omega\in\Omega_7$ and $p\in\mathbb{R}^d\setminus U$, for $v^\delta$ the solution of (\ref{macroscopic}) corresponding to $p\in\mathbb{R}^d$, \begin{equation}\label{application_bounded_6}\lim_{\delta\rightarrow 0}\sup_{x\in B_{R/\delta}}\abs{\overline{H}(p)+\delta v^\delta(x,\omega)}=0.\end{equation}\end{prop}

\begin{proof}  Write $\tilde{H}:\mathbb{R}^d\times\mathbb{R}^d\times\Omega\rightarrow\mathbb{R}$ for the Hamiltonian satisfying (\ref{steady}), (\ref{application_equal}) and (\ref{application_convex}).  For each $\delta>0$, $p\in\mathbb{R}^d$ and $\omega\in\Omega$, we write $\tilde{v}^\delta$ for the solution of the macroscopic problem \begin{equation}\label{application_bounded_1} \delta \tilde{v}^\delta+\tilde{H}(p+D\tilde{v}^\delta,x,\omega)=0\;\;\textrm{on}\;\;\mathbb{R}^d.\end{equation}  We will prove that there exists $R>0$ such that, for every $\abs{p}\geq R$, we have $\tilde{v}^\delta=v^\delta$ on $\mathbb{R}^d$, for $v^\delta$ the solution of (\ref{macroscopic}) corresponding to $p\in\mathbb{R}^d$.

Observe that (\ref{coercive}), (\ref{bounded}), (\ref{hamcon}) and (\ref{application_equal}) imply that there exists $\overline{\alpha}\in\mathbb{R}$ such that, for every $\alpha\geq\overline{\alpha}$, $y\in\mathbb{R}^d$ and $\omega\in\Omega$, \begin{equation}\label{application_bounded_2} \left\{\;q\in\mathbb{R}^d\;|\;H(q,y,\omega)\leq\alpha\;\right\}=\left\{\;q\in\mathbb{R}^d\;|\;\tilde{H}(q,y,\omega)\leq\alpha\;\right\}, \end{equation} and, \begin{equation}\label{application_bounded_3} \left\{\;q\in\mathbb{R}^d\;|\;H(q,y,\omega)\geq\alpha\;\right\}=\left\{\;q\in\mathbb{R}^d\;|\;\tilde{H}(q,y,\omega)\geq\alpha\;\right\}. \end{equation}

Notice that the comparison principle, see \cite{CIL}, implies that, for every $p\in\mathbb{R}^d$, $\omega\in\Omega$ and $\delta>0$, for $v^\delta$ the solution of (\ref{macroscopic}) corresponding to $p\in\mathbb{R}^d$, \begin{equation}\label{application_bounded_4} \inf_{(y,\omega)\in\mathbb{R}^d\times\Omega}H(p,y,\omega)\leq -\delta v^\delta\leq \sup_{(y,\omega)\in\mathbb{R}^d\times\Omega}H(p,y,\omega)\;\;\textrm{on}\;\;\mathbb{R}^d.\end{equation}  In view of (\ref{coercive}), choose $R>0$ such that, for every $\abs{p}\geq R$, \begin{equation}\label{application_bounded_5}\inf_{(y,\omega)\in\mathbb{R}^d\times\Omega}H(p,y,\omega)\geq\overline{\alpha}.\end{equation}

Fix $p\in\mathbb{R}^d\setminus B_R$ and let $v^\delta$ denote the solution of (\ref{macroscopic}) corresponding to $p$.  We will prove that $v^\delta$ satisfies (\ref{application_bounded_1}).  In view of (\ref{macroscopic}) and (\ref{application_bounded_2}), since (\ref{application_bounded_4}) implies that $$-\delta v^\delta\geq\overline{\alpha}\;\;\textrm{on}\;\;\mathbb{R}^d,$$ whenever $v^\delta-\phi$ has a local maximum at $x_0\in\mathbb{R}^d$ for $\phi\in\C^2(\mathbb{R}^d)$, $$p+D\phi(x_0)\in\left\{\;q\in\mathbb{R}^d\;|\;H(q,x_0,\omega)\leq-\delta v^\delta(x_0,\omega)\;\right\}=\left\{\;q\in\mathbb{R}^d\;|\tilde{H}(q,x_0,\omega)\leq-\delta v^\delta(x_0,\omega)\;\right\},$$ and, similarly, whenever $u-\phi$ has a local minimum at $x_0\in\mathbb{R}^d$ for $\phi\in \C^2(\mathbb{R}^d)$, $$p+D\phi(x_0)\in\left\{\;q\in\mathbb{R}^d\;|\;H(q,x_0,\omega)\geq-\delta v^\delta(x_0,\omega)\;\right\}=\left\{\;q\in\mathbb{R}^d\;|\tilde{H}(q,x_0,\omega)\geq-\delta v^\delta(x_0,\omega)\;\right\}.$$  We therefore conclude that $v^\delta$ satisfies $$\delta v^\delta+\tilde{H}(p+Dv^\delta,y,\omega)=0\;\;\textrm{on}\;\;\mathbb{R}^d.$$

Since $p\in\mathbb{R}^d\setminus B_R$ was arbitrary, for every $p\in\mathbb{R}^d\setminus B_R$ uniqueness, see \cite{CIL}, implies that $v^\delta=\tilde{v}^\delta$.  In view of Theorem \ref{homogenization_proof} and (\ref{application_convex}), this implies that there exists a subset $A_1\subset\Omega$ of full probability such that (\ref{application_bounded_6}) is satisfied for every $\omega\in A_1\cap \Omega_6$, for every $\abs{p}\geq R$.

In order to conclude, recall the definition of $\tilde{H}(p)$ in Proposition \ref{sure_ball}, and define the subset of full probability $$\Omega_7=A_1\cap\Omega_6\;\;\textrm{and}\;\;\textrm{U}=\left\{\;p\in\mathbb{R}^d\;|\;\tilde{H}(p)\neq \overline{H}(p)\;\right\}.$$  Since a repetition of the argument appearing in Proposition \ref{effective_properties} proves that $\tilde{H}\in\Lip(\mathbb{R}^d)$, the subset $U$ is open with, by the above, $U\subset B_R$.  This completes the argument.  \end{proof}

We now prove that this situation can be considered generic, insofar as the general question of homogenization for first-order equations is concerned.  We assume that, for every Hamiltonian $\hat{H}:\mathbb{R}^d\times\mathbb{R}^d\times\Omega\rightarrow\mathbb{R}$ satisfying (\ref{steady}), (\ref{application_equal}) and (\ref{application_convex}) there exists a subset $\Omega_8=\Omega_8(\hat{H})\subset\Omega$ such that, for every $\omega\in\Omega_8$, $p\in\mathbb{R}^d$ and $R>0$, for $\hat{v}^\delta$ the solution of \begin{equation}\label{application_generic} \delta \hat{v}^\delta+\hat{H}(p+D\hat{v}^\delta,y,\omega)=0\;\;\textrm{on}\;\;\mathbb{R}^d, \end{equation} we have \begin{equation}\label{application_generic_1} \lim_{\delta\rightarrow 0}-\delta \hat{v}^\delta(0,\omega)\;\;\textrm{exists.}\end{equation}  Observe that, in view of Proposition \ref{sure_ball} and Theorem \ref{homogenization_proof}, (\ref{application_generic_1}) is equivalent to the homogenization of (\ref{intro_eq}) corresponding to $\hat{H}$.

\begin{prop}\label{application_equivalent}  Assume (\ref{steady}).  Furthermore, assume that for every Hamiltonian $\hat{H}:\mathbb{R}^d\times\mathbb{R}^d\times\Omega\rightarrow\mathbb{R}$ satisfying (\ref{steady}), (\ref{application_equal}) and (\ref{application_convex}) there exists a subset $\Omega_8=\Omega_8(\hat{H})\subset\Omega$ of full probability satisfying (\ref{application_generic_1}).  There exists a subset $\Omega_9\subset\Omega$ of full probability such that, for every $w\in\Omega_9$, $p\in\mathbb{R}^d$ and $R>0$, for $v^\delta$ the solution of (\ref{macroscopic}) corresponding to $p\in\mathbb{R}^d$, $$\lim_{\delta\rightarrow 0}\sup_{x\in B_{R/\delta}}\abs{\overline{H}(p)+\delta v^\delta(x,\omega)}=0.$$  \end{prop}

\begin{proof}  We proceed by contradiction.  Suppose that there exists $p_0\in\mathbb{R}^d$ such that, for $v^\delta$ the solution of (\ref{macroscopic}) corresponding to $p_0$, for every $\omega\in\Omega_4$, \begin{equation}\label{application_equivalent_1} \liminf_{\delta\rightarrow 0}-\delta v^\delta(0,\omega)=\tilde{H}(p_0)<\overline{H}(p_0)=\limsup_{\delta\rightarrow 0}-\delta v^\delta(0,\omega).\end{equation}  In view of (\ref{bounded}), fix $\overline{\alpha}\in\mathbb{R}$ satisfying \begin{equation}\label{application_equivalent_2} \overline{\alpha}>\sup_{(y,\omega)\in\mathbb{R}^d\times\Omega}H(p_0,y,\omega).\end{equation}  We will now define a Hamiltonian $\hat{H}:\mathbb{R}^d\times\mathbb{R}^d\times\Omega\rightarrow\mathbb{R}$ satisfying (\ref{steady}), (\ref{application_equal}) and (\ref{application_convex}) such that, for $\hat{v}^\delta$ the solution of (\ref{application_generic}) corresponding to $p_0$, we have $v^\delta=\hat{v}^\delta$.

In view of (\ref{coercive}), fix $R>0$ such that for every $y\in\mathbb{R}^d$ and $\omega\in\Omega$, \begin{equation}\label{application_equivalent_4} \left\{\;q\in\mathbb{R}^d\;|\;H(q,y,\omega)\leq\overline{\alpha}\;\right\}\subset B_R.\end{equation}

We define $\hat{H}:\mathbb{R}^d\times\mathbb{R}^d\times\Omega\rightarrow\mathbb{R}$ by the rule \begin{equation}\label{application_equivalent_3} \hat{H}(p,y,\omega)=\left\{\begin{array}{ll} H(p,y,\omega) & \textrm{if}\;\;(p,y,\omega)\in \left\{\;H(p,y,\omega)\leq \overline{\alpha}\;\right\}, \\ \overline{\alpha} & \textrm{if}\;\;(p,y,\omega)\in\left(B_R\times\mathbb{R}^d\times\Omega\right)\setminus\left\{\;H(p,y,\omega)\leq \overline{\alpha}\;\right\},  \\ \abs{p}^2-R^2+\overline{\alpha} & \textrm{if}\;\;(p,y,\omega)\in\left(\mathbb{R}^d\setminus B_R\right)\times\mathbb{R}^d\times\Omega.  \end{array}\right.\end{equation}  It follows from the definition that $\hat{H}$ satisfies (\ref{steady}), (\ref{application_equal}) and (\ref{application_convex}) for the deterministic $\tilde{H}:\mathbb{R}^d\times\mathbb{R}^d\times\Omega\rightarrow\mathbb{R}$ defined by $$\tilde{H}(p,y,\omega)=\abs{p}^2-R^2+\overline{\alpha}.$$  It remains to prove that $\hat{v}^\delta=v^\delta$, for $\hat{v}^\delta$ the solution of (\ref{application_generic_1}) corresponding to (\ref{application_equivalent_3}).

In view of (\ref{application_equivalent_2}) and (\ref{application_equivalent_3}), for every $\alpha<\overline{\alpha}$, for every $y\in\mathbb{R}^d$ and $\omega\in\Omega$, \begin{equation}\label{application_equivalent_5} \left\{\;q\in\mathbb{R}^d\;|\;H(q,y,\omega)\leq\alpha\;\right\}=\left\{\;q\in\mathbb{R}^d\;|\;\hat{H}(q,y,\omega)\leq\alpha\;\right\}, \end{equation} and, \begin{equation}\label{application_equivalent_6} \left\{\;q\in\mathbb{R}^d\;|\;H(q,y,\omega)\geq\alpha\;\right\}=\left\{\;q\in\mathbb{R}^d\;|\;\hat{H}(q,y,\omega)\geq\alpha\;\right\}. \end{equation}  The remainder of the argument now follows identically to Proposition \ref{application_bounded}, in view of (\ref{application_bounded_4}) and (\ref{application_equivalent_2}), which contradicts (\ref{application_generic_1}) for every $\omega\in\Omega_8(\hat{H})\cap\Omega_6.$  We then define $\Omega_9(p_0)=\Omega_8(\hat{H})\cap\Omega_6$.

Since $p_0\in\mathbb{R}^d$ was arbitrary,  we conclude that, for each $p\in\mathbb{R}^d$, there exists a subset $\Omega_9=\Omega_9(p)\subset\Omega$ of full probability satisfying the proposition's conclusion.  To conclude, define the subset of full probability $$\Omega_9=\bigcap_{q\in\mathbb{Q}^d}\Omega_9(p),$$ and apply Proposition \ref{macroscopic_continuity}.  \end{proof}

\section{Appendix}

We present here a collection of well-known results that play a role in the proof.  Our aim is to present simple proofs that apply in our setting.  In most cases, somewhat more general statements are available, and references are provided.

The first proposition proves that a process with a stationary, mean zero gradient is necessarily strictly sub-linear at infinity.  The presentation follows what can be found in the appendix of \cite{AS}, which proves a somewhat more general result.

\begin{prop}\label{appendix_sublinear}  Let $z:\mathbb{R}^d\times\Omega\rightarrow\mathbb{R}$ be a process satisfying, on a subset of full probability, $z(\cdot,\omega)\in\Lip(\mathbb{R}^d)$.  Let $Z:\mathbb{R}^d\times\Omega\rightarrow\mathbb{R}^d$ be a stationary ergodic process satisfying, on a subset of full probability, $Z(\cdot,\omega)=Dz(\cdot,\omega)$ in the sense of distributions with $$\mathbb{E}\left(Z(0,\omega)\right)=0\;\;\textrm{and}\;\;\mathbb{E}\left(\abs{Z(0,\omega)}\right)<\infty.$$  Then, on a subset of full probability, $$\lim_{\abs{y}\rightarrow 0}\frac{z(y,\omega)}{\abs{y}}=0.$$ \end{prop}

\begin{proof}  The stationarity, ergodicity and integrability of $Z(y,\omega)$ implies by the ergodic theorem, see Theorem \ref{ergodic1}, that there exists a subset $A_1\subset\Omega$ of full probability such that, for each bounded, open set $V\subset\mathbb{R}^d$ containing the origin, for each $\omega\in A_1$, \begin{equation}\label{effective_sublinear_1} \lim_{\epsilon\rightarrow 0}\dashint_{V}Z(x/\epsilon,\omega)\;dy=\mathbb{E}\left(Z(0,\omega)\right)=0.\end{equation}

We define, for each $\omega\in\Omega$, $z^\epsilon(x,\omega)=\epsilon z(x/\epsilon,\omega)$ and observe that there exist a subset $A_2\subset A_1$ of full probability such that, for each $\omega\in A_2$, the family $\left\{z^\epsilon(\cdot,\omega)\right\}_{\epsilon>0}$ is uniformly Lipschitz continuous.

Fix $\omega\in A_2$.  Then, after passing to a subsequence $\left\{\epsilon_k\rightarrow 0\right\}_{k=1}^\infty$, for $\overline{z}\in \Lip(\mathbb{R}^d)$, we have $$z^{\epsilon_k}\rightarrow\overline{z}\;\;\textrm{locally uniformly on}\;\;\mathbb{R}^d,$$  and, in the sense of distributions, for each $R>0$, $$Dz^{\epsilon_k}\rightharpoonup D\overline{z}\;\;\textrm{weakly in}\;\;L^2(B_R).$$  In view of $\omega\in A_2\subset A_1$ and (\ref{effective_sublinear_1}), this implies that $$D\overline{z}=0\;\;\textrm{and}\;\;\overline{z}(0)=\lim_{k\rightarrow\infty}\epsilon_k z(0,\omega)=0,$$  and, therefore, since $\overline{z}\in\Lip(\mathbb{R}^d)$, that $\overline{z}=0$.

Since $\omega\in A_2$ and the subsequence $\left\{\epsilon_k\rightarrow 0\right\}_{k=1}^\infty$ were arbitrary, we have, for all $\omega\in A_2$, as $\epsilon\rightarrow 0$, $$z^\epsilon(\cdot,\omega)\rightarrow 0\;\;\textrm{locally uniformly on}\;\;\mathbb{R}^d.$$  After returning to the original scaling, this completes the proof.  \end{proof}

The following proposition proves that distributional derivatives are preserved, almost surely, in the limit.  We use this fact in Proposition \ref{metric_limit} to prove that, for each $p\in\mathbb{R}^d$, the limit $\overline{m}$ is nonnegative.

\begin{prop}\label{appendix_distributional}  Fix $R>0$ and $1<p<\infty$.  Suppose that $\left\{z_k\right\}_{k=1}^\infty\subset L^p(B_R\times\Omega)$ and $\left\{Z_k\right\}_{k=1}^\infty\subset L^p(B_R\times\Omega;\mathbb{R}^d)$ are such that, on a subset of full probability, for every $1\leq k< \infty$, $Z_k(\cdot,\omega)=Dz_k(\cdot,\omega)$ in the sense of distributions.  If $$z_k\rightharpoonup z\;\;\textrm{weakly in}\;\;L^p(B_R\times\Omega)$$ and $$Z_k\rightharpoonup Z\;\;\textrm{weakly in}\;\;L^p(B_R\times\Omega;\mathbb{R}^d)$$ then, on a subset of full probability, $$Z(\cdot,\omega)=Dz(\cdot,\omega)\;\;\textrm{in the sense of distributions.}$$ \end{prop}

\begin{proof}  Let $1<q<\infty$ satisfy $1/p+1/q=1$.  Fix a countable, dense subset $\left\{\phi_k\right\}_{k=1}^\infty$ of $L^q(B_R)$.  Then, by Fubini's theorem, the weak convergence implies that, for every measurable subset $B\subset\Omega$, we have the equality of vectors $$\int_B\int_{B_R}zD\phi_k\;dxd\mathbb{P}=-\int_B\int_{B_R} \phi_k Z\;dxd\mathbb{P}.$$  This implies that, for each $1\leq k< \infty$, there exists a subset of full probability $A_k\subset\Omega$ satisfying, for all $\omega\in A_k$, $$\int_{B_R}z(x,\omega)D\phi_k(x)\;dx=-\int_{B_R}\phi_k(x)Z(x,\omega)\;dx.$$

Define the subset of full probability $$\overline{A}=\bigcap_{k=1}^\infty A_k.$$  The density of the $\left\{\phi_k\right\}_{k=1}^\infty$ implies that, for all $\omega\in\overline{A}$, $$Z(\cdot,\omega)=Dz(\cdot,\omega)\;\;\textrm{in the sense of distributions,}$$ which completes the proof.  \end{proof}

The final proposition of the appendix describes the gradient of a Lipschitz continuous subsolution to a spacially independent, first-order Hamilton-Jacobi equation following a standard regularization by convolution.  The statement applies in general.  We choose to use the effective Hamiltonian $\overline{H}:\mathbb{R}^d\rightarrow\mathbb{R}$ only to avoid introducing additional notation.

In what follows, we use $\rho_\epsilon:=\epsilon^{-d}\rho(x/\epsilon)$ to denote a rescaling of a standard, compactly supported, nonnegative, radially symmetric, smooth convolution kernel $\rho:\mathbb{R}^d\rightarrow[0,\infty)$.

\begin{prop}\label{appendix_regular}  Assume (\ref{steady}).  Fix $\alpha\in\mathbb{R}$.  Let $z\in\Lip(\mathbb{R}^d)$ satisfy \begin{equation}\label{appendix_regular_1} \overline{H}(Dz)\leq \alpha\;\;\textrm{on}\;\;\mathbb{R}^d.\end{equation}  For each $\epsilon>0$, if $z^\epsilon=z\ast \rho_\epsilon,$ $$Dz^\epsilon(x)\in\textrm{Conv}\left(\left\{\;\overline{H}(q)\leq \alpha\;\right\}\right)\;\;\textrm{for all}\;\;x\in\mathbb{R}^d.$$\end{prop}

\begin{proof}  Fix $\epsilon>0$.  Since, by definition, $\textrm{Conv}\left(\left\{\;\overline{H}(q)\leq \alpha\;\right\}\right)$ is convex, it suffices to prove that every linear inequality satisfied by the elements of $\textrm{Conv}\left(\left\{\;\overline{H}(q)\leq \alpha\;\right\}\right)$ is satisfied $Dz^\epsilon(x)$, for every $x\in\mathbb{R}^d$.

In view of $z\in\Lip(\mathbb{R}^d)$, Rademacher's theorem implies that $z$ is differentiable almost everywhere and satisfies (\ref{appendix_regular_1}) classically at every point of differentiability.  Suppose that for $q\in\mathbb{R}^d$ and $\beta\in\mathbb{R}$ we have, for every $x\in\textrm{Conv}\left(\left\{\;\overline{H}(q)\leq \alpha\;\right\}\right)$, $$x\cdot q\leq \beta.$$  This implies that, at every point of differentiability, $$Dz\cdot q\leq \beta.$$  And, therefore, for every $x\in\mathbb{R}^d$, $$Dz^\epsilon(x)\cdot q=\int_{\mathbb{R}^d}\rho_\epsilon(x-y)Dz(y)\cdot q\;dy\leq \beta\int_{\mathbb{R}^d}\rho_\epsilon(x-y)\;dy=\beta.$$  Since $q\in\mathbb{R}^d$, $\beta\in\mathbb{R}$ and $\epsilon>0$ were arbitrary, this completes the proof.  \end{proof}

\bibliography{nonconvex}
\bibliographystyle{plain}

\end{document}